\documentclass[12pt]{article}

\usepackage{hyperref}       
\usepackage{url}            
\usepackage{booktabs}       
\usepackage{amsfonts}       
\usepackage{nicefrac}       

\usepackage{amssymb,amsbsy,amsmath,amscd,amsthm,amsfonts,MnSymbol}
\usepackage{cite}
\usepackage{enumerate}
\usepackage{dsfont}
\usepackage{graphicx}
\usepackage{color}
\usepackage{subfigure}

\allowdisplaybreaks

\usepackage[square,numbers]{natbib}
\newtheorem{theorem}{Theorem}[section]
\newtheorem{lemma}[theorem]{Lemma}
\newtheorem{corollary}[theorem]{Corollary}
\newtheorem{proposition}[theorem]{Proposition}

\newtheorem{remark}{Remark}[section]
\newtheorem{example}{Example}[section]

\newcommand\numberthis{\addtocounter{equation}{1}\tag{\theequation}}

\def\R{\mathbb{R}}

\def\N{\mathbb{N}}
\def\E{\mathbb{E}}
\def\P{\mathbb{P}}

\def\eps{\epsilon}
\def\del{\delta}

\def\la{\lambda}
\def\t{\theta}
\def\a{\alpha}

\def\d{\mathrm{d}}

\def\c{{\complement}}

\def\h{\hat}

\def\hi{h(X_i(\t))}
\def\M{\mathfrak{M}_n(\t)}
\def\m{\mathfrak{M}_n}
\def\g{\mathfrak{g}}

\def\hm{\hat{\mu}_{n}}
\def\hl{\hat{\lambda}}

\def\sgn{\mathrm{sgn}}

\def\hw{\hat{w}}
\def\I{\mathcal{I}}
\def\c{{\complement}}

\newcommand{\ip}{\mc{I}_+}
\newcommand{\ic}{\mc{I}_-}
\newcommand{\iz}{\mc{I}_0}
\newcommand{\mc}[1]{\mathcal{#1}}

\renewcommand{\numberthis}{\addtocounter{equation}{1}\tag{\theequation}}
\renewcommand{\l}[0]{\left }
\renewcommand{\r}[0]{\right}

\makeatletter
\renewcommand*{\@cite@ofmt}{\hbox}
\makeatother

\begin{document}
\title{Maximum Likelihood under constraints: Degeneracies and Random Critical Points}
\author{
	\begin{tabular}{c}
		{Subhroshekhar Ghosh}
		 \footnote{
 Department of Mathematics,
National University of Singapore, Singapore 119076}\\subhrowork@gmail.com
	\end{tabular}
	\and
	\begin{tabular}{c}
		{Sanjay Chaudhuri}
		 \footnote {
 Dept. of Stat. \& Applied Prob.,
 National University of Singapore, Singapore 117546}\\
  {stasc@nus.edu.sg} \\
	\end{tabular}
}
\date{}
\maketitle

%
%
%
\begin{abstract}
We investigate the problem of semi-parametric  maximum likelihood under constraints on summary statistics. Such a procedure results in a discrete probability distribution that maximises the likelihood among all such distributions under the specified constraints (called estimating equations), and is an approximation to the underlying population distribution. The study of such \textit{empirical likelihood} originates from the seminal work of Owen (\citep{owen_1988}, \citep{owenbook}). We investigate this procedure in the setting of mis-specified (or biased) estimating equations, i.e. when the null hypothesis is not true.  We establish that the behaviour of the optimal distribution under such mis-specification differ markedly from their properties under the null, i.e. when the estimating equations are unbiased and correctly specified.
    This is manifested by certain ``degeneracies'' in the optimal distribution which define the likelihood.  Such degeneracies  are not observed  under the null.
    Furthermore, we establish an anomalous behaviour of the log-likelihood based Wilks'
    statistic, which, unlike under the null, \textit{does not} exhibit a chi-squared limit.
    In the Bayesian setting, we rigorously establish the posterior consistency of procedures based on these ideas, where instead of a parametric likelihood, an empirical likelihood is used to define the posterior distribution.
    In particular, we show that this posterior, as a random probability measure, rapidly converges to the delta measure at the true parameter value.
  A novel feature of our approach is the investigation of critical points of random functions in the context
    of such empirical likelihood. In particular, we obtain the location and the mass of the degenerate optimal weights as the leading and sub-leading terms in a canonical expansion of a particular critical point of a random function
that is naturally associated with the model.

  \end{abstract}

\section{Introduction}
In this paper, we investigate the problem of maximum likelihood under constraints on summary statistics. We will consider the likelihood over the space of discrete probability distributions supported on the given data points. We will consider such a distribution that maximises this likelihood, under the constraint that the expectation of certain summary statistics based on the candidate distribution should match specified values, related to what is believed to be  the corresponding expectation under the population distribution. These  equations are referred to as the estimating equations. 
The end result of this procedure is a candidate probability distribution that is supported on the observed data points on one hand, and well-approximates the population distribution on the other. For a detailed discussion of the procedure, we refer the reader to Section \ref{sec:notpd}. In summary, this is a likelihood-based method  to find the ``best approximation'' to the law of the population based on empirical data (incorporated through the data-dependent estimating equations),  and hence is referred to as the method of empirical likelihood.

Empirical likelihood is a popular paradigm in
semi-parametric statistics, applicable to a wide range of
scientific scenarios.  
It was introduced in the seminal work  of Owen (\citep{owen_1988}, also see \citep{owenbook}), who used it for testing statistical hypothesis and established the asymptotic properties of the corresponding log-likelihood based Wilks' statistic under the null.  Since then, several authors have deduced various properties of this likelihood and
the parameter estimates obtained by maximising it.  Over the years a body of literature has emerged involving this concept, touching upon several diverse areas as: hypothesis testing \citep{diciccio_hall_romano_1991}, parameter estimation \citep{qin_lawless_1994}, density estimation \citep{hall_owen_1993}, large deviations \citep{mykland_1999,grendar_judge_2009,kitamura_santos_shaikh_2012},
semiparametric inference \citep{murphy_1995,murphy_vandervaart_1997}, survival analysis \citep{li_1995}, high dimensional inference \citep{hjort_mckeague_keilegom_2009,peng_schick_2013,chang_tang_wu_2018}, regression \citep{chen2009review}, to provide a partial list.

Most of the known theoretical properties of empirical likelihood are deduced at or near the true value of the parameter, for which the expectations of the parametric constraints are zero.  Much less studied are its properties under mis-specified constraints, that is, when the true expectations of the estimating equations specifying the constraints, do not equal zero.

In this article we investigate the properties of empirical likelihood under mis-specified constraints.  We show that the properties of the optimal weights which define the likelihood are markedly different from their properties under correct specification. Such understanding of the ``landscape'' would have implications vis-a-vis   the behaviour of statistical procedures based on the empirical likelihood paradigm, some of which we already explore in this article. Further investigations in this vein have been carried out in the related methodological paper \citep{chaudhuri2018easy}.

Mis-specified constraints can affect the likelihood in two ways. First of all, the constrained optimisation problem may be infeasible, in which case it is  customary to define the empirical likelihood to be zero.  When the sample size is  sufficiently large,
the likelihood remains positive even under mis-specification. However, as we demonstrate below, the optimal weights which are used to define the maximum likelihood distribution exhibit certain degeneracies which are not observed under the null.  We first explore the location and the mass of such degeneracies.
Similar degeneracies have been known to researchers in statistical physics \citep{jaynes_1957a,jaynes_1957b} and exponential random graph models (abbrv. ERGM) \citep{snijders_pattison_robins_handcock_2006,robins_pattison_kalish_lusher_2007,chatterjee_diaconis_2013,horvat_czabarke_toroczkai_2015,fellows_handcock_2017}.  Our results apply to such situations and, in fact, provide a quantification of such phenomena.
Our investigations crucially hinge on a certain canonical expansion that we derive for the Lagrange multiplier connected to the constraints in the optimisation problem.  Such expansions under mis-specification appear to be unknown in the literature,
and can be a powerful tool, as we demonstrate in the Bayesian setting (see Section \ref{sec:postcons}).

Even though the properties of the log-likelihood based Wilks' statistic that corresponds to the empirical likelihood are well known under the null hypothesis, its properties under the alternative are less clear. For details on the Wilks' statistic, we refer the reader to Section \ref{sec:wilks}.  Assuming mild conditions, \citep{owenbook}  proves that, under the truth, the Wilk's statistic asymptotically follows a chi-squared distribution.
In view of this result, as in the case for classical parametric Wilks' statistics, it is perhaps natural to assume that under the alternative the Wilks' statistic would asymptotically be distributed as a non-central chi-square random variable.
\textcolor{black}{Under appropriate local alternatives, collapsing to the null at certain rate, this is indeed true \citep{qin_lawless_1994,lazar_mykland_1998,kitamura_santos_shaikh_2012}. 
In this paper we show that, interestingly, that is not always the case. In fact, under a fixed alternative (i.e., an alternative that does not collapse to the null as $n \to \infty$), the Wilks' statistic asymptotically would not converge to any distribution, and furthermore, does not exhibit any discernible scaling limit behaviour}.

  Properties of empirical likelihood under mis-specified constraints are extremely useful for its application in the Bayesian setting, where the usual parametric likelihood is replaced by the empirical likelihood evaluated at a given value of the parameter.
  Direct use of empirical likelihood in a Bayesian setting goes back to \citep{lazar_2003}.  Many authors have considered the use of Bayesian empirical likelihood (BayesEL) and similar procedures in recent times.
  Some properties of such procedures have been studied by \citep{schennach_2005}, \citep{fang_mukerjee_2005, fang_mukerjee_2006}, \citep{chaudhuri_mondal_yin_2017}, \citep{chib_shin_simoni_2018}, \citep{grendar_judge_2009}, \citep{zhong_ghosh_2016}, \citep{vexler_zou_hutson_2018}, \citep{yuan_clarke_2010}, among others.
  Domains of application include small area estimation \citep{chaudhuri_ghosh_2011,porter_holan_wikle_2015},  quantile regression \citep{yang_he_2012}, approximate Bayesian computation \citep{mengersen_pudlo_robert_2012} etc.
Posterior consistency of the BayesEL procedures have been discussed by \citep{grendar_judge_2009}, employing  sophisticated results from the theory of large deviations.  In this work we provide a short and succinct proof of fast rate of convergence of a BayesEL posterior distribution.

\textcolor{black}{
  We use the nature of critical points of random functions to investigate the properties of empirical likelihood under mis-specification.  
Investigation of random critical points have attracted attention in probability in the recent years (see, e.g., \citep{auffinger2013random}, \citep{auffinger2013complexity}, and the references therein).
  Our \textit{ab initio}  approach allows us to obtain a novel canonical expansion for the optimal value of the Lagrange multiplier for the underlying constrained optimization problem. This optimal Lagrange parameter is actually related to the critical point of a natural random function associated with the model.  In fact, we obtain an understanding the precise order of growth of the optimal Lagrange multiplier. 
  The successive terms in the canonical expansion yield the location and the mass of the degenerate optimal weights, which is one of our main results.  A related expansion of the log-likelihood based Wilks' statistic under mis-specification, that allows us to deduce its asymptotic properties, follows from these considerations.
}


\section{Set up and Model}
\subsection{Notations and Problem Description}
\label{sec:notpd}
Suppose $X_1, \ldots, X_n$ are i.i.d. random variables taking values in some space $E$.  The common distribution of the $X_i$-s is $F_\gamma$, which depends on a parameter $\gamma\in\Theta\subseteq\R^q$. 
In the spirit of non-parametric statistics, no knowledge of an analytic form of the distribution $F_\gamma$ is assumed.  Let $h : E \times \Theta \mapsto \R$ be a function such that, for any given $\vartheta \in\Theta$, we have
\begin{equation}\label{eq:esteq}
\E_{F_{\vartheta}}\left[h(X_1,\gamma)\right]=0 \text{ if } \gamma=\vartheta.
\end{equation}

A simple example of the above set-up is where the $X_i$-s themselves take values in $\R^q$ (in other words, $E=\R^q$), the distribution $F_\gamma$ has mean $\gamma$, and the function $h(X_1,\gamma)=X_1-\gamma$.

Let $\t_0$ be given and fixed. Let $\t_0$ be the \textit{ground truth}, that is, $(X_1,\ldots,X_n)$ are generated from the probability distribution $F_{\t_0}$. For $\t \in \Theta$, we consider the null hypothesis is that the parameter value is equal to $\t$. When this $\t$ (as in the null hypothesis) happens to equal the ground truth $\t_0$, we say that we are analysing the testing problem \textit{under the null} (or, \textit{under the truth}). In this case, the problem is said to be \textit{correctly specified}. When $\t$ (as in the null hypothesis) is not the same as the ground truth $\t_0$ (and $\E_{F_{\theta_0}}\left[h(X_1,\t)\right] \ne 0 $), the problem is said to be \textit{mis-specified}, and the estimating equations are said to be \textit{biased}.   

Denoting $w=\left(w_i\right)^n_{i=1}$, we compute:
\begin{equation}\label{eq:emplik}
\hat{w}=\mbox{argmax}_{w \in \mathcal{W}}\prod^n_{i=1}w_i,
\end{equation}
where
\begin{equation} \label{eq:constraint}
\mathcal{W}=\left\{w~:~\sum^n_{i=1}w_i h(X_i,\t)=0\right\}\cap\Delta_{n-1}
\end{equation}
and $\Delta_{n-1}$ is the standard simplex in $\R^n$, given by
\[
\Delta_{n-1}=\left\{w~:~w_i\ge 0,~i=1,2,\ldots,n,~\sum^n_{i=1}w_i=1\right\}.
\]

This optimization problem is feasible if
\begin{equation}\label{eq:feas}
\min_{i=1,2,\ldots,n}h\left(X_i,\t\right)<0<\max_{i=1,2,\ldots,n}h\left(X_i,\t\right).
\end{equation}
If the optimization problem is not feasible we define $\hat{w}=0$.

Once $\hat{w}$ has been computed, we can obtain an approximation of the underlying distribution of the $X_i$-s by considering the atomic measure which puts mass $\h{w}_i$ at $X_i$, $i=1$, $2$, $\ldots$, $n$.  This leads to an approximation of the underlying 
probability distribution via
\begin{equation}\label{eq:mhat}
  \d \hat{F}_n(x)=\sum^n_{i=1}\hat{w}_i \del_{X_i}(x),
\end{equation}
where $\del_{a}$ is the Dirac delta mass at the point $a$.

It may be noted that $\hat{F}_n$ can be considered to be a constrained empirical estimate of the underlying distribution $F_\t$.  In the absence of the constraint imposed by \eqref{eq:constraint}, i.e. when $\mathcal{W}=\Delta_{n-1}$, it turns out that $\hat{w}_i=n^{-1}$, for
all $i=1$, $2$, $\ldots$, $n$ and the corresponding $\hat{F}_n$ is the well-known empirical measure.


When $\theta=\theta_0$ (i.e., when the null hypothesis is true), the constraint \eqref{eq:constraint} with $w=\hat{w}$ (which can be abbreviated as $\E_{\hat{F}_n}[h(X,\t)]=0$) is an empirical version of the relation \eqref{eq:emplik} with $\vartheta=\t_0$, which captures the ``ground truth'' that holds true in the population. As such, this situation is also referred to as the correctly specified case.  In such a situation, the properties of the optimal weight vector $\hat{w}$ and $\hat{F}_n$ are well known.  
In particular, \citep{owenbook} shows that under the truth, asymptotically, the corresponding log-likelihood based Wilks' statistic, given by $-2\sum^n_{i=1}\log(n\hat{w}_i)$, follows a chi-squared distribution with one degree of freedom.

The focus of our investigations in the present paper is the case where the null hypothesis entails that the parameter value is $\t\in\Theta$, which happens to be  such that $\t \ne \t_0$ (where $\t_0$ is the true parameter value) and
\begin{equation}\label{eq:miseq}
\E_{F_{\t_0}}\left[h(X_1,\t)\right]\ne 0.
\end{equation}
Under this situation, the constraint \eqref{eq:constraint}  does not have a counterpart that holds true for the population. The constraint, therefore, is a mis-representation of the ground truth in this setting. As such, this scenario is referred to as the ``mis-specified case''.

Although, for large $n$, even under mis-specification the optimization problem \eqref{eq:constraint} would be feasible, that is the condition \eqref{eq:feas} would be satisfied, many of the ideas and  methods that are valid under the null would no longer be applicable.

Throughout the rest of this paper, a sequence of events $\{E_n\}_{n}$ is said to occur \textit{with high probability} (abbrv. \textit{w.h.p.}) if $\P(E_n) \to 1$ as $n \to \infty$.

\subsection{An Illustrative Example}
\begin{figure}[h!]

\centering
        \resizebox{2.5in}{2.5in}{\includegraphics{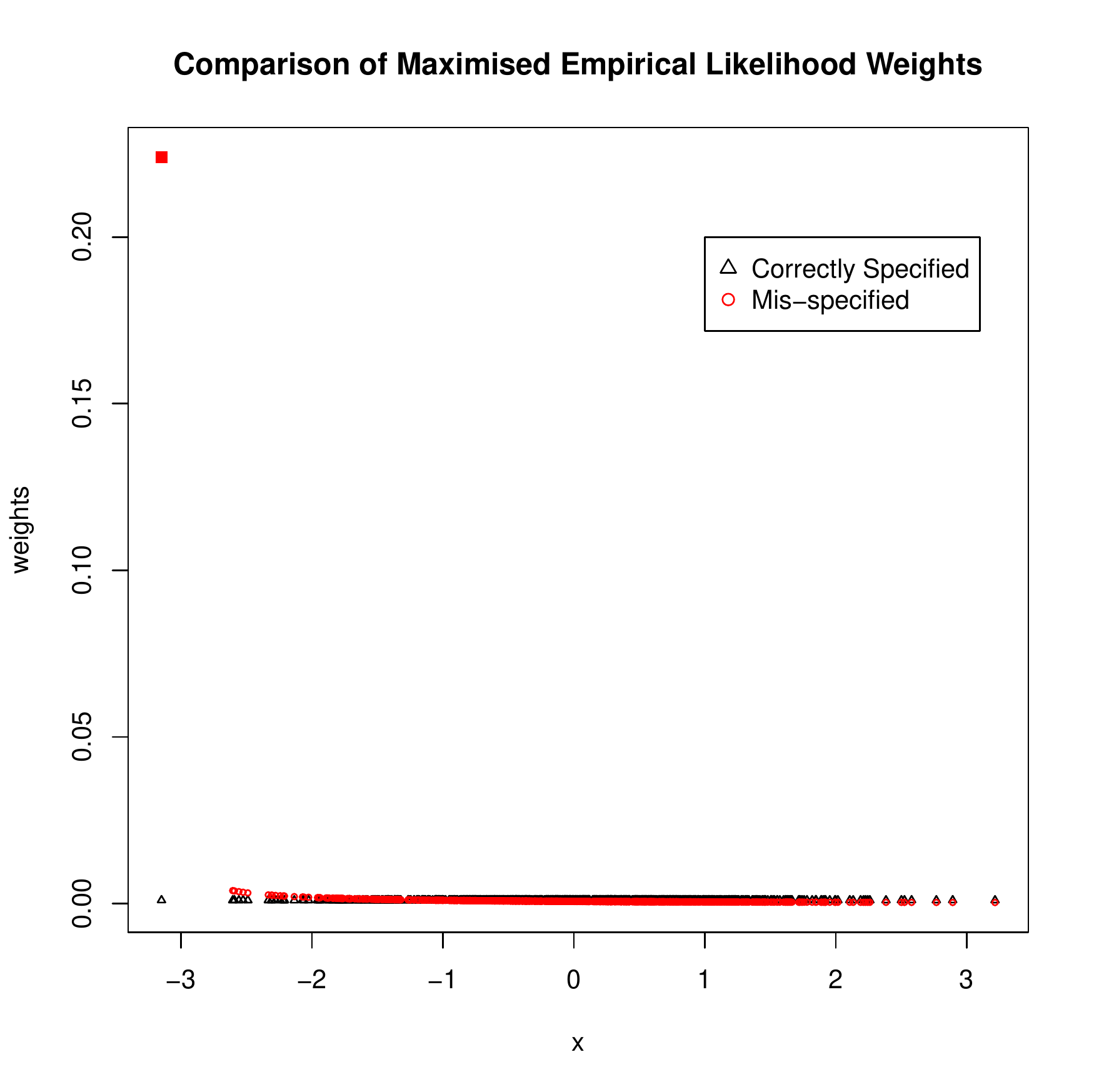}}
    \caption{Gaussian observations with correctly specified mean (black) vs. mis-specified mean (red). Note the degenerate mass at the top left, under mis-specification.}
    \label{fig:Normal}
\end{figure}

\begin{example}\label{ex:gaussian}
In order to illustrate the statistical behaviour of the maximised weights under mis-specified estimating equation, we consider estimating the empirical distribution function from $n=1000$, i.i.d. $N\left(0,1\right)$ observations. We are interested in understanding the mean of the underlying distribution. We are dealing with a location family where the observable $h$ is given by $h(X,\t)=X-\t$.

When the null hypothesis is that the mean is 0, the estimating equation is correctly specified. The resulting maximised weights $\hat{w}$ has been plotted in black in the Figure \ref{fig:Normal} above.  It is clearly, seen that each $\hat{w}_i$ is close to $.001$ (ie. 1/n) and there are no large weights.

If, however, the null hypothesis is that the mean is -1, then the estimating equation is mis-specified. This is a scenario where the null hypothesis is not true. The nature of the optimal weights $\hat{w}$ changes.  In Figure \ref{fig:Normal}, the weight corresponding to the most negative observation (marked by the red square) is significantly, roughly $249$ times, larger than its weight in the correctly specified case.
  This is the largest weight we get when the constraint is mis-specified.  All other optimal weights are small. The largest weight is about $59$ times larger than the second largest weight.
  Note that, even under mis-specification \eqref{eq:feas} holds, that is, the problem in \eqref{eq:emplik} is still feasible under the mis-specified constraints.  Furthermore, this discrepancy in the optimal weights cannot be explained by finiteness of the sample size.
Even if the sample size is increased, such large optimal weight corresponding to one of the extreme observations is obtained.


\end{example}

From the Example above, it clearly follows that mis-specified constraints as in \eqref{eq:miseq} leads to differences in the statistical behaviour of optimal weights, which leads to results of a very different flavour.
It is this difference in statistical behaviour under the alternative that we will investigate below.

\subsection{Preliminary Considerations}

We maximise the product \eqref{eq:emplik}, constrained by the first order moment constraint \eqref{eq:constraint}. 

For the null hypothesis given by $\t\in\Theta$, we abbreviate
\[
h_i:=h(X_i,\t)
\]
and consider the objective,
\begin{equation} \label{eq:lagrangian}
L(w,\la,\a)=\sum_{i=1}^n \log nw_i - \a(\sum_{i=1}^n w_i-1) - n \la \sum_{i=1}^n w_i h_i,
\end{equation}
where $\a$ and $\la$ are Lagrange multipliers respectively ensuring that the constraints that $\sum^n_{i=1} w_i=1$ and $\sum_{i=1}^n w_i h_i=0$ are  satisfied. 

Let $(\hw,\hat{\la},\hat{\a})$ is the argmax for this optimization problem. Setting $\frac{\partial L}{\partial w_i} \big|_{\hw,\hl,\hat{\a}}=0$, we deduce that $\frac{1}{\hw_i} =\hat{\a}+n \hl h_i$ for every $i$. Multiplying both sides by $\hw_i$, summing over $i$ and using the constraints $\sum_{i=1}^n \hw_i h_i=0$ and $\sum_{i=1}^n \hw_i =1$ gives $\hat{\a}=n$. Substituting $\hat{\a}=n$ back into $\frac{1}{\hw_i} =\hat{\a}+n \hl h_i$ and re-arranging, we obtain 
\begin{equation} \label{eq:MLE}
\hw_i=\frac{1}{n} \frac{1}{1+\hl h_i}.
\end{equation}
Substituting \eqref{eq:MLE} into the constraint $\sum_{i=1}^n \hw_i h_i =0$, it quickly follows that $\hl$ satisfies the equation
\begin{equation} \label{eq:hleq}
\sum_{i=1}^n \frac{h_i}{1+\hl h_i}=0.
\end{equation}

For more detailed considerations, we refer the reader to  \citep{owenbook} .

A key implicit quantity which significantly impacts the statistical behaviour in this model is the Lagrange multiplier $\la$ (or rather, its optimal value $\hl$), and indeed our first result will be about the asymptotic behaviour of $\hl$.
The information about the asymptotic properties of $\hl$ will then be exploited to obtain an understanding of the optimal weights and the associated Wilks' statistic.

A key aspect of our approach is the introduction of the perspective of critical points of random functions in the study of empirical likelihood. Let us consider the function \[\mathcal{L}(\la) =-\sum_{i=1}^ n \log (1+ \la h_i), \] for any $\la$ such that $1+ \la h_i>0$ for all $i$.
It follows from \eqref{eq:MLE} that, up to a non-random additive factor, $\mathcal{L}(\hl)$ is the maximum value of the log of the objective in \eqref{eq:emplik}. Since the objective in \eqref{eq:emplik} is the likelihood of the weight vector $w$, we deduce that $\mathcal{L}(\hl)$ equals, in fact, the maximum log-likelihood for this model.

Furthermore, \eqref{eq:hleq} implies that $\hl$ is a solution to the equation $\mathcal{L}'(\la)=0$, which means that $\hl$ is a critical point of the function $\mathcal{L}$. In fact, the critical value of the function $\mathcal{L}$ at $\la=\hl$ is the maximum value of the log-likelihood for our model, as noted above.

The constraint $1+\hl h_i > 0$ for each $i$, a consequence of \eqref{eq:MLE} and the fact that $0\le \hw_i \le 1$ for each $i$.



\section{Main Results}\label{sec:main}
\textcolor{black}{
We denote $a(\t)=\E_{F_{\t_0}}[h(X_i,\t)]$, and for each $i=1$, $2$, $\ldots$, $n$, we define the random errors $\xi_i(\t)$ by
\[
h_i=h(X_i,\t)=a(\t)+\xi_i(\t).
\]
From the basic set-up of our model, we clearly have $a(\t)=0$ if $\t=\t_0$. On the other hand, we are interested in mis-specified case, where $a(\t)\ne 0$. 
    Our results are applicable to a very wide class of distributions of the random errors $(\xi_i(\t))^n_{i=1}$ satisfying a few mild distributional assumptions (valid for uniformly for all $\t \in \Theta$) that we list below.
}

\begin{itemize}
\item[(A1)] (The errors) \textcolor{black}{The random variables $\xi_1(\t)$, $\xi_2(\t)$, $\ldots$, $\xi_n(\t)$ are i.i.d. with zero mean and finite second moment.  We will denote the common distribution to be that of the random variable $\xi(\t)$.}
\item[(A2)] (The rate of growth of maxima) There is a non-random sequence $\m \to \infty$ such that $\m=o(n)$ and:
 \begin{itemize}
\item[(1)] $\frac{1}{\m}\max_{1\le i \le n} |\xi_i(\t)| 1_{\xi_i(\t)>0} = 1 + o_P(1)$, as $n \to \infty$ .
\item[(2)]  $\frac{1}{\m}\max_{1\le i \le n} |\xi_i(\t)| 1_{\xi_i(\t)<0} = 1 + o_P(1)$, as $n \to \infty$ .
 \end{itemize}
\item[(A3)] (Separation of first and second maxima) Let us denote: 
\[\xi^+_{\max}(\t):=\max_{1\le i \le n} \xi_i(\t) 1_{\xi_i(\t)>0} \hspace{10pt} \text{and} \hspace{10pt} \xi^-_{\max}(\t):=\max_{1\le i \le n} \xi_i(\t) 1_{\xi_i(\t)<0},\]
and \[\xi^+_{\max,2}(\t):=\max_{1\le i \le n} \{\{\xi_i(\t) 1_{\xi_i(\t)>0}\} \mathbin{\backslash} \{ \xi^+_{\max}(\t) \} \}, \]  \[\xi^-_{\max,2}(\t):=\max_{1\le i \le n} \{\{\xi_i(\t) 1_{\xi_i(\t) <0}\} \mathbin{\backslash} \{ \xi^-_{\max}(\t) \} \}.\]
We assume that, \[|\xi^+_{\max}(\t) - \xi^+_{\max,2}(\t)| \ge \m^{-\gamma} \hspace{10pt} \text{and} \hspace{10pt} |\xi^-_{\max}(\t) - \xi^-_{\max,2}(\t)| \ge \m^{-\gamma}\] with probability $\to 1$ as $n \to \infty$, where $\gamma$ and $\m$ satisfy the condition that $\m^{\gamma+2}=o_P(n)$.
\item[(A4)] (Decay of tails) Let $p_{n,\del}:=\P[|\xi(\t)|> \m^{1-\del}]$, for $\del>0$. Then, for small enough $\del>0$,\textcolor{black}{we assume} $\m^{\gamma+2} \log n \cdot p_{n,\del} \to 0$ as $n \to \infty$.
\end{itemize}

\textcolor{black}{Rates of growths of the extreme order statistics and spacings have been studied by several authors.  We specifically refer to \citep{deheuvels_1986, nagaraja_1982, welsch_1973, reiss2012approximate} for results on general distributions.  Many distributions would satisfy the Assumptions (A1)-(A4).  Of course, in the most crucial class of models the random offsets $\xi_i$ are assumed to be Gaussian random variables. 
Our assumptions are satisfied by the standard Gaussian distribution for $\m=\sqrt{2 \log n}$, any $\gamma>1/2$ and any $\del<1/2$.} It is known that, under very general conditions, the spacing between the largest and the second largest order statistics is essentially of the same order as the fluctuations of the maximum (see \citep{nagaraja_1982} and \citep{nagaraja2015spacings}, in particular equation (33) in the latter). These fluctuations are related to the \textsl{norming constants} of a distribution that are well known in extreme value theory (see \citep{resnick2013extreme}, in particular Chapter 1 therein). It can be seen from these considerations that our conditions would be satisfied by a wide class of error distributions $\xi_i$. A typical non-Gaussian example would be the Laplace distribution (i.e. the difference of two i.i.d. exponential random variables), which arises in important settings (see, e.g., \citep{kotz2012laplace} and the references therein). More generally, using the tools of \citep{nagaraja2015spacings} and \citep{resnick2013extreme}, we can obtain $\m,\gamma$ and $\delta$
(as in the assumptions (A1)-(A4)) for specific distributions.

\subsection{Asymptotic behaviour of $\hl$ under mis-specification}\label{sec:hl}
We start by deriving an asymptotic expansion of the estimate of the Lagrange multiplier $\lambda$, when the constraint has been mis-specified.
To that end,  we introduce the notation
  \[
  \mathrm{sgn}(x)=\begin{cases}
  -1&if~x<0,\\
  1&if~x>0,\\
  0&if~x=0.
    \end{cases}
  \]
  We first show that:

\begin{theorem} \label{thm : lagrange}
Under (A1)-(A4), for any $\t \in \Theta$ such that $a(\t) \ne 0$, we have
\[
\m \hl = \sgn(a(\t)) + o_P(1).
\]
\end{theorem}

\textcolor{black}{We contrast Theorem \ref{thm : lagrange} with the situation under correct specification.  Assuming finite third moment, when $a(\t)=0$, \citep{owenbook} shows that $\hl=O_p(n^{-1/2})$. On the other hand, under mis-specification, Theorem \ref{thm : lagrange} demonstrates that $\hl$ is typically of the order of $1/\sqrt{\log n}$ when the offsets are Gaussian.}

Using the input from  Theorem \ref{thm : lagrange}, we obtain the following asymptotic expansion of $\hl$:
\begin{theorem} \label{thm : expansion}
Under (A1)-(A4), for any $\t \in \Theta$ such that $a(\t)\ne 0$,  we have
\begin{equation} \label{eq : expansion}
\hl=\frac{\sgn(a(\t))}{\m}-\frac{a(\t)^{-1}}{n}+o_P\left(\frac{1}{n}\right).
\end{equation}
\end{theorem}

\textcolor{black}{Asymptotic properties of the optimal weights and that of the corresponding Wilks' statistic depend crucially on the asymptotic properties of $\hl$. Next we discuss these properties under mis-specification.}

\subsection{Degeneracy in the Optimal Weights}\label{sec:wd}
We now establish the existence of degeneracies in the \textcolor{black}{optimal weights} under mis-specification, and obtain a description of their location and their \textcolor{black}{magnitude} in terms of the leading and sub-leading terms \textcolor{black}{in the asymptotic expansion of $\hl$ in} \eqref{eq : expansion}.

{
\begin{theorem} \label{thm : degeneracy}
  Suppose, for any $\t \in \Theta$ such that $a(\t) \ne 0$ we define,

\begin{equation}\label{eq:hmax}
    h_{\max}=\begin{cases}
    \min_{1\le i\le n} h_i,&\text{if $a(\t)>0$,}\\
    \max_{1\le i\le n} h_i,&\text{if $a(\t)<0$.}
    \end{cases}
  \end{equation}
  Let $\hw_{\max}$ be the optimal weight on $h_{\max}$ in the solution of mis-specified problem in \eqref{eq:emplik}, with the index $i_{\max}$ being such that $h_{i_{\max}}=h_{\max}$.  Then under (A1) - (A4) we have
\begin{equation} \label{eq : deg1}
\hw_{\max}=\frac{|a(\t)|}{\m},
\end{equation}
and furthermore,
\begin{equation} \label{eq : deg2}
\max\{\hw_i~:~i\ne i_{\max}\}=O_P\left(\frac{\m^{\gamma+1}}{n}\right).
\end{equation}
We also have
\begin{equation} \label{eq : deg3}
\min\{\hat{w}_i~:~1\le i \le n\}\ge \frac{1}{n}(1-o_P(1)).
\end{equation}
\end{theorem}
}

\begin{remark} \label{rem:hmax}
For $n$ large enough, with high probability, there are $h_i$-s with positive as well as negative sign (since $a$ is fixed and $\xi_i$-s assume large values both in the positive and the negative directions, by assumption (A2)). Since there are going to be $h_i$-s of both signs, $h_{\max}$ as above will be well defined, and will have the opposite sign of $a$. In particular, for $a(\t)>0$, we have $\sgn(a(\t))>0$ and hence $h_{\max}<0$, and in this case we have we have $h_{\max}=-|h_{\max}|$.
\end{remark}

\textcolor{black}{Theorem \ref{thm : degeneracy} completely characterises the optimal weights and the corresponding empirical estimate of the distribution of $X_1$, $X_2$, $\ldots$, $X_n$ obtained by maximising \eqref{eq:emplik} under the mis-specified constraint in \eqref{eq:constraint}.} \textcolor{black}{Condition (A3) entails, in particular, that $\m^{\gamma + 2}=o(n)$, which implies that under mis-specification $\max\{\hw_i~:~i\ne i_{\max}\}=o_P(\hw_{\max})$, that is $\hw_{max}$ is significantly larger than the rest of the weights. This clearly stands in contrast with the correctly specified case ($\t=\t_0$), because in that case it is known (\citep{owenbook}, \citep[Lemma $1$.]{yuan_xu_zheng_2014}) that under mild assumptions on the offsets $\xi_i$, $\hw_i=O_p(1/n)$, for all $i=1$, $2$, $\ldots$, $n$.} \eqref{eq : deg3} completes the picture, by demonstrating that, under mis-specification, the minimum optimal weight is not too far below the average optimal weight, that is, $1/n$.

\textcolor{black}{As an illustration we reconsider Example \ref{ex:gaussian}.  Recall that, under mis-specification, $a(\t)=1$ and as predicted in Theorem \ref{thm : degeneracy} the largest weight $\hw_{\max}$ corresponds to minimum observation.
}

\textcolor{black}{The empirical estimate of the population distribution behaves very differently in the mis-specified case:
the estimated measure $\hm$ has an atom at the observation which corresponds to $h_{\max}$ with an unusually large weight compared to rest of the distribution.  For Gaussian observations (see Example \ref{ex:gaussian}), its ratio with the maximum of the rest of the optimal weights blows up linearly with $n$ (up to logarithmic factors).}


\subsection{Wilks' Statistic under mis-specification}\label{sec:wilks}
\textcolor{black}{One defines the log-likelihood based Wilks' statistic as:
  \[
  \mathbf{L}= - 2 \sum_{i=1}^n \log n \h{w}_i .
  \]
This statistic plays an important role in hypothesis testing in this context, with the test rejecting the null hypothesis if $\mathbf{L}$ is large.
}

\textcolor{black}{When the null hypothesis is true (i.e. $\t = \t_0$), it is well-known (\citep{owenbook}) that 
  \[
  -2\sum^n_{i=1}\log(n\hat{w}_i)=\frac{\left\{\sum^n_{i=1}h(X_i)\right\}^2}{\sum^n_{i=1}\left\{h(X_i)\right\}^2}+o_p(1).
  \]
That is, $\mathbf{L}$ has a $\chi^2(1)$ distribution in the limit as $n \to \infty$.}

\textcolor{black}{In order to calculate the asymptotic power of the test, the behaviour of $\mathbf{L}$ under mis-specification, i.e. when $\t\ne\t_0$ has to be considered.
  In view of the asymptotic property under the null, it is natural to envisage that for $\t\ne\t_0$, the Wilks' statistic would converge to a non-central $\chi^2(1)$ distribution with the non-centrality parameter depending on $a(\t)$.
Under the local alternatives of the form $\t_1=\t_0+u/\sqrt{n}$, this is indeed true.  In particular, for such alternatives, \citep{qin_lawless_1994} prove the local asymptotic normality of the empirical likelihood ratio statistic holds.
A more detailed evaluation of the power of the empirical likelihood tests under this class of local alternatives can be found in \citep{lazar_mykland_1998}.  Using results from the theory of large deviations,
\citep{kitamura_santos_shaikh_2012} discuss asymptotic optimality of empirical likelihood for testing moment condition for a large class of alternatives.}

\textcolor{black}{However, due to the degeneracies in the optimal weights, for a fixed alternative not collapsing to $\t_0$, the asymptotic behaviour of $\mathbf{L}$ is quite anomalous.
  In fact, we show that its asymptotic distribution is far from a non-central chi-square. We prove this result in the Theorem below.}

\begin{theorem} \label{thm : Wilks}
Under (A1)-(A4), for $\t \in \Theta$ such that $a(\t)\ne 0$ we have
\[ \mathbf{L}=\frac{2n}{\m}|a(\t)|(1+o_P(1)).  \]
\end{theorem}

For a large class of $\xi$ such that all its moments exist (for instance, the Gaussian distribution), we can extend the asymptotic expansion in Theorem \ref{thm : Wilks} to any degree. This is encapsulated by
\begin{corollary} \label{cor:Wilks}
Under the conditions of Theorem \ref{thm : Wilks} and the additional assumption that all moments $\mu_j=\E[h(X_1)^j]<\infty$, for any fixed $k \in \N$ we have
\[ \mathbf{L}= 2n \cdot \l[\sum_{j=1}^k (-1)^{j-1} \sgn(a(\t))^j \frac{\mu_j}{j \m^j}\r] \cdot (1+o_P(1)) \]
\end{corollary}
\textcolor{black}{Thus, under a fixed alternative, the Wilks' statistic $\mathbf{L}$, in fact, does not appear to exhibit any discernible distributional scaling limit as $n \rightarrow \infty$.  It diverges at the rate of $n/\M$ and the power of the test commensurately grows to one. }

\section{The Bayesian perspective}
\label{sec:postcons}

We exhibit the power of our approach by employing it to obtain a succinct proof of the posterior consistency of the so called Bayesian empirical likelihood procedures,  abbreviated as BayesEL procedures.

The BayesEL procedure entails starting with a prior $\pi$ on $\t\in \Theta$ (a compact subset of $\R^d$), and estimating equations defining the empirical likelihood.  For any given value of $\t\in\Theta$, the problem in \eqref{eq:emplik} is numerically solved to obtain the optimal weights.
Once these weights $\hw$ are determined, the posterior density is given by:
\begin{equation}\label{eq:post}
 \Pi^{(n)}(\t):=\Pi\left(\t\mid X_1,\ldots,X_n\right)=\frac{e^{\sum^n_{i=1}\log(\hw_i)}\pi(\t)}{\int_{\Theta}e^{\sum^n_{i=1}\log(\hw_i)}\pi(\t)d\t}.
\end{equation}
If for certain $\t$, the problem in \eqref{eq:emplik} is infeasible, the posterior is defined to be zero.

This posterior cannot be expressed in a closed form. Thus for the purpose of statistical inference, samples are drawn from it using specialised Markov Chain Monte Carlo procedures \citep{chaudhuri_mondal_yin_2017}.



We show that the random measure $\Pi^{(n)}(\t)$ converges to the delta measure $\del_{\t_0}$ as $n \to \infty$, where $\t_0$ is the true value of the parameter.
While there has been previous work on the posterior consistency of Bayesian empirical likelihood (see, e.g., \citep{grendar_judge_2009}), most of the known approaches rely heavily on sophisticated results and methods from large deviations theory.
Our approach, although simple and self-contained, provides fast rates of convergence. In particular, for Gaussian errors, we show an exponential decay up to logarithmic factors.

This is the content of Theorem \ref{thm : consistent} below. We prove our results in a slightly more general setting than above, which 
would require a slightly different set of assumptions ((B1)-(B6)) than we have been operating with so far, and these will be detailed in Section \ref{sec : consistent}.

\begin{theorem} \label{thm : consistent}
Under (B1)-(B7), as $n \to \infty$, the random measure $\Pi^{(n)}(\t)$ converges in probability to $\del_{\t_0}$ on the space of probability measures on $\R^d$.
\end{theorem}

Theorem \ref{thm : consistent} shows that as $n\rightarrow\infty$, for any bounded continuous function $\mathfrak{f}$ and for any open ball $B\subset\mathbb{R}^d$ containing $\t_0$, we have $\int_{B^C}\mathfrak{f}(\t)\Pi^{(n)}(\t)d\t\rightarrow 0$.
This is accomplished by, first, by finding an upper bound of $\exp\l(\sum^n_{i=1}\hw_i\r)$ in the numerator of \eqref{eq:post}, roughly giving an exponential decay in $n$ outside the set $B$ (up to a logarithmic factor), and then by finding a sub-exponential lower bound of the denominator.

\section{Degeneracies in Statistical Networks}\label{sec:graph}
\textcolor{black}{An important example of a mis-specified constraint can be found in the statistical analysis of social networks.  Here the observations $X_i$ are i.i.d. samples from a set of random graphs $G_N$ with $N$ vertices, and $h$ a real-valued observable on the graph, e.g. the (centred) triangle count.
  The parameter $\t$ can be taken to be the connection probability for each edge.}

\textcolor{black}{A popular class for statistical models for social networks are the so called exponential random graph models (ERGM).  Given a observed value of the observable $h$, say $h_0$, these models assign probabilities to each graph $X_i$ by maximising the entropy.  More precisely, the probability $\h{v}=(\h{v}_i)^n_{i=1}$ are given by:
  \begin{equation}\label{eq:et}
  \h{v}=\mathrm{arg}\max_{v\in\mathcal{V}}\sum^n_{i=1}-v_i\log v_i,
  \end{equation}
  where
  \begin{equation}\label{eq:graphcons}
  \mathcal{V}=\l\{v~:~\sum^v_{i=1}v_ih_i=h_0\r\}\cap\Delta_{n-1}.
  \end{equation}
}
\textcolor{black}{  Via a simple computation, one can show that for each $i=1$, $2$, $\ldots$, $n$, we have $\hat{v}_i=\mbox{exp}(\kappa(h_i-h_0))/\sum^n_{i=1}\mbox{exp}(\kappa(h_i-h_0))$, where the Lagrange multiplier $\kappa$ satisfy the equation $\sum^n_{i=1}(h_i-h_0)\mathrm{exp}(\kappa(h_i-h_0))=0$.}

\textcolor{black}{Estimating a distribution function by entropy maximisation has a long history. The procedure described above was introduced by \citep{jaynes_1957a,jaynes_1957b} and seen many applications in statistics and econometrics.  Some recent examples include, \citep{schennach_2005}, \citep{chaudhuri_ghosh_2011}, \citep{chib_shin_simoni_2018} etc.
  Both entropy maximisation and empirical likelihood can be viewed to be minimising two limiting divergences in the Cressie-Read family \citep{cressie_read_1984}.  Their asymptotic properties under the correct specification are known to be similar \citep{kitamura_2001}.}

\begin{figure}[h!]

  \centering
  \begin{subfigure}[Empirical Likelihood\label{fig:distGrEl}]{
      \resizebox{2.4in}{2.4in}{\includegraphics{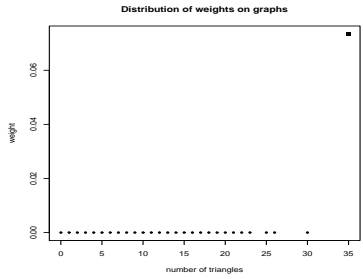}}}
      \end{subfigure}\qquad
      \begin{subfigure}[Entropy Maximisation\label{fig:distGrEt}]{
      \resizebox{2.4in}{2.4in}{\includegraphics{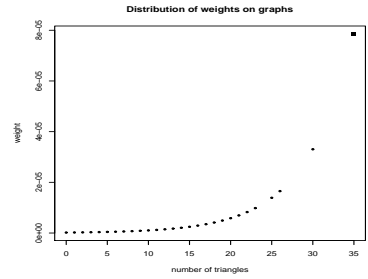}}}
      \end{subfigure}
      \begin{subfigure}[Empirical Likelihood \label{fig:margGrEl}]{
      \resizebox{2.4in}{2.4in}{\includegraphics{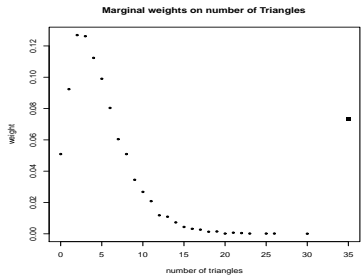}}}
      \end{subfigure}\qquad
      \begin{subfigure}[Entropy Maximisation \label{fig:margGrEt}]{
      \resizebox{2.4in}{2.4in}{\includegraphics{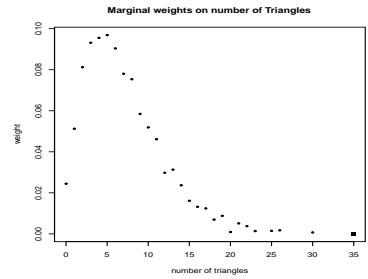}}}
      \end{subfigure}
      \caption{Comparison of Empirical Likelihood and maximum entropy methods}
      \label{fig:graph}
  \end{figure}

\textcolor{black}{  In the present situation however, since the expectation of $h_1$ will hardly be equal to $h_0$, in general, by construction, the constraints in \eqref{eq:graphcons} are mis-specified.  In ERGM literature (see, e.g. \citep{chatterjee_diaconis_2013}), it is known that degeneracies tend to appear, in the sense that the optimal weight distribution tends to put large weights on nearly empty or nearly complete graphs or both.  Such degeneracies are considered to be inconvenient for statistical inference, specially for model interpretation, and several authors have considered possibilities for avoiding them \citep{horvat_czabarke_toroczkai_2015,fellows_handcock_2017}.
}
\textcolor{black}{  Results in Section \ref{sec:main} indicate that the same degeneracies would be observed if the probabilities on the observed graphs are estimated using empirical likelihood instead, that is, if $\hat{v}$ is obtained by maximising the $\prod^n_{i=1}v_i$ over $\mathcal{V}$.}

\textcolor{black}{  We compare the degrees of degeneracies which result from the above two procedures through a simple simulation study below.}

\textcolor{black}{    Suppose we consider a complete enumeration of Erdos-Renyi $G(N,\frac{1}{2})$ graphs on $N=7$ nodes. 
The constraint function $h$ is taken to be the triangle count of the graph.  The observed count $h_0$ is kept fixed at $7$.}

\textcolor{black}{    The results are displayed in Figure \ref{fig:graph} above.  The weights on the individual graphs obtained from empirical likelihood and entropy maximisation are displayed in Figures \ref{fig:distGrEl} and \ref{fig:distGrEt} respectively.  From these figures it is clear that the
    degree of degeneracy of the distribution obtained from the empirical likelihood is much higher than that of the distribution obtained by entropy maximisation. 
}
\textcolor{black}{    Figures \ref{fig:margGrEl} and \ref{fig:margGrEt} point towards another interesting property.  In these figures, the total probabilities of a particular numbers of triangles are displayed.  No degeneracy is seen when these marginal probabilities obtained from entropy maximisation, whereas the marginal probabilities obtained from empirical likelihood are still degenerate.
    This absence of degeneracy in Figure \ref{fig:margGrEt} is explained by the fact that the frequencies of the triangle counts form a log-concave function \citep{horvat_czabarke_toroczkai_2015}.  Similar result for the procedure based on empirical likelihood is unknown and appear to be far more complex.
}

\section{Discussion}\label{sec:discussion}


We conduct most of our investigations in the setting of constraints that are one dimensional  (i.e., the constraint function $h$ is real valued). The crucial exception to this is the Bayesian setting, and
our analysis of the BayesEL procedures actually hold in higher dimensions with additional mild assumptions. 

Extending some of our main results to higher dimension would not be straightforward.  Our asymptotic expansions use critical points of random functions, whose behaviour is poorly understood in higher dimensions.
In fact it turns out that, the behaviour of optimal weights for multi-dimensional constraints are much more complex.  We illustrate this fact using a simple example.

Suppose we generate $n=1000$ observations from a bivariate normal random variable $(X,Y)$ with zero means, unit standard deviations and correlation equal to $0.5$.  A plot of the optimal weights against the norm of the observations computed under two constraints $\sum^n_{i=1}w_iX_i=0$ and $\sum^n_{i=1}w_iY_i=0$ is presented in Figure \ref{fig:truth}.
The optimal weights seem to be confined in cone-shaped space with linear boundaries.  With small mis-specification, in particular, $\sum^n_{i=1}w_i(X_i-0.1)=0$ and $\sum^n_{i=1}w_i(Y_i+0.1)=0$ in Figure \ref{fig:smallMS}), those weights are still confined, however, the boundaries turn out to be non-linear.


\begin{figure}[h!]
  \centering
  \begin{subfigure}[True Constraints. \label{fig:truth}]{
      \resizebox{.475\columnwidth}{2.75in}{\includegraphics[angle=90,origin=c]{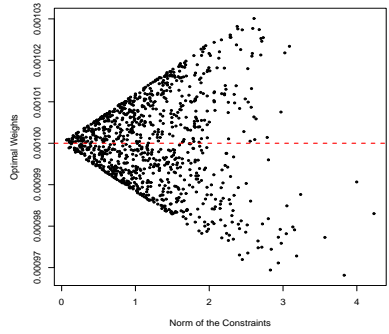}}}
  \end{subfigure}
  \begin{subfigure}[Small Mis-specification. \label{fig:smallMS}]{
      \resizebox{.475\columnwidth}{2.75in}{\includegraphics[angle=90,origin=c]{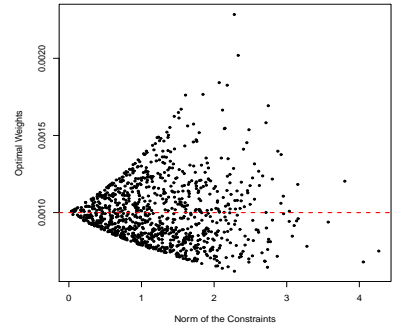}}}
  \end{subfigure}
  \begin{subfigure}[Big Mis-specification: Scatterplot.\label{fig:bigMSscatter}]{
      \resizebox{.475\columnwidth}{2.75in}{\includegraphics{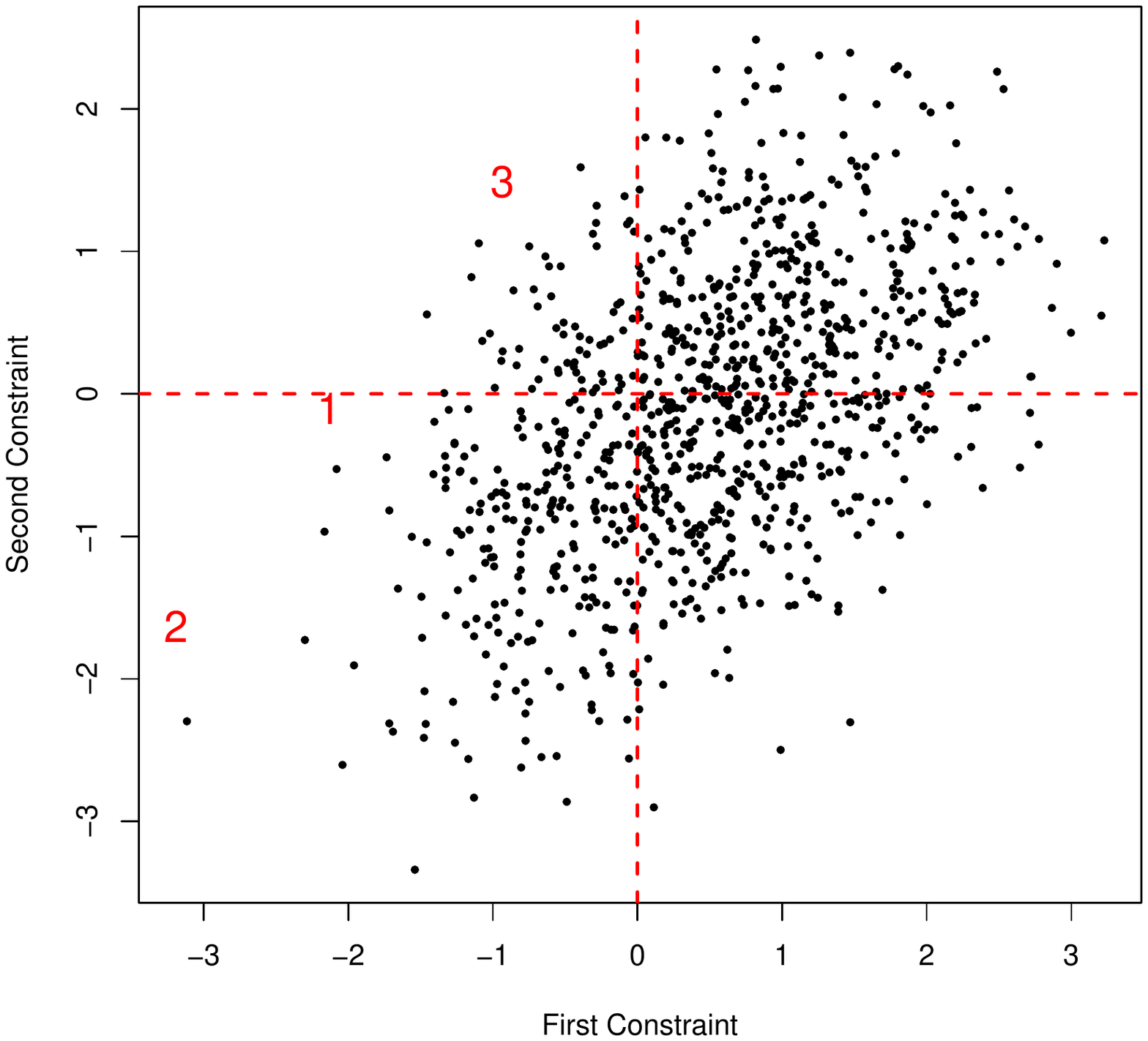}}}
  \end{subfigure}
  \begin{subfigure}[Big Mis-specification: Weights vs Norm.\label{fig:bigMSnorm}]{
      \resizebox{.475\columnwidth}{2.75in}{\includegraphics{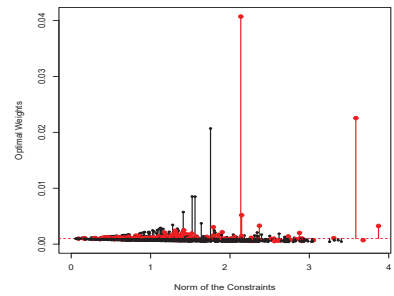}}}
  \end{subfigure}
  \caption{EL Degeneracy in higher dimension}
  \label{fig:multi}
 \end{figure}

With larger mis-specifications, the situation gets more complex.  In Figures \ref{fig:bigMSscatter} and \ref{fig:bigMSnorm}, the weights are computed under the constraints $\sum^n_{i=1}w_i(X_i+0.5)=0$ and $\sum^n_{i=1}w_i(Y_i-0.1)=0$.  In the scatter plot presented in Figure \ref{fig:bigMSscatter} the numbers $1$ to $3$,
indicate the positions of the observations with three largest weights, which turn out to be larger than $0.01$ (i.e. $10/n$).
From Figure \ref{fig:bigMSnorm}, it is clear that there are more than one weights which are relatively large compared to the rest.  The position of the observations with these degenerate weights are not easily determined either.
Clearly, the largest two weights are in the third quadrant.  However, none of them correspond to the constraint with highest norm, not even among the observations in the third quadrant (indicated in Figure \ref{fig:bigMSnorm} in red).

The results in our work also have methodological implications for the analysis of real data. A case in point is the application of empirical likelihood techniques in the context of approximate Bayesian computation. The details of this application are the subject of a related methodological paper \citep{chaudhuri2018easy}. 

\section{The order of the Lagrange multiplier $\hl$}

In this section, we establish that the correct order of the Lagrange parameter $\la$ is $\m^{-1}$, thereby proving Theorem \ref{thm : lagrange}.
In order to facilitate the proof, we will first introduce some definitions. The main proof will invoke certain auxiliary lemmas, which we will state first and prove later in the section. Throughout this section, we will work with a fixed $\t \in \Theta$, and abbreviate $a(\t)$ and $\xi_i(\t)$ simply as $a$ and $\xi_i$ respectively. 

Define the function
\begin{equation} \label{eq:g}
  g(\la)=\sum^n_{i=1}\frac{h_i}{1+\la h_i}.
 \end{equation}
Clearly, we have,
\begin{equation} \label{eq:functiong}
g(0)=\sum^n_{i=1}h_i=na+\sum^n_{i=1}\xi_i=n\left(a+\frac{\sum^n_{i=1}\xi_i}{n}\right).
\end{equation}
By the Law of Large Numbers, $\sum^n_{i=1}\xi_i/n\rightarrow 0$ almost surely as $n\rightarrow\infty$, we have $\sgn(g(0))=\sgn(a)$ for all large enough $n$. Further, we note that $\hl$ is clearly a solution of $g(\la)=0$.

We will also define three sets of indices from amongst $\{1,\dots,n\}$:

\begin{align}
\ip&=\left\{i~:~h_i>0\right\},\label{eq:ip}  \\
\ic&=\left\{i~:~h_i<0\right\},\label{eq:ic} \\
\text{and } \iz&=\left\{i~:~|h_i|>\m^{1-\delta} + |a|\right\}. \label{eq:iz}
\end{align}
By definition, $i\in\iz$ implies $|\xi_i|>\m^{1-\delta}$. Now $|\iz|=\sum^n_{i=1}1_{\{h_i\in\iz\}}$. Hence, we have
\[
\E\left[|\iz|\right]=\sum^n_{i=1}\P\left[h_i\in\iz\right]=\sum^n_{i=1}\P\left[|\xi_i|>\m^{1-\delta}\right]\le np_{n,\del}.
\]
To summarize,

\begin{equation}
\E\left[\frac{|\iz|}{n}\right]\le p_{n,\del}.
\end{equation}
\begin{remark} \label{rem:i0order}
This, in particular, implies that the random variable $|\I_0|/n$ is $O_P(p_{n,\del})=o_P(1)$.
\end{remark}

We now state two auxiliary lemmas that will be useful in the proof of Theorem \ref{thm : lagrange}.

\begin{lemma} \label{lem:interval}
Suppose $a(\t)>0$. Then the optimal Lagrange parameter $\hat{\lambda}$ belongs to the interval $(0,1/|h_{\max}|)$ with high probability. Furthermore, the function $g$, as defined in \eqref{eq:g}, is monotonically decreasing in this interval.
\end{lemma}

\begin{lemma} \label{lem:stick}
Suppose $a(\t)>0$. For any fixed $\epsilon>0$ (independent of $n$), if
$0 < \lambda \le \frac{1-\epsilon}{|h_{\max}|},
$
then $g(\lambda)>0$ with high probability.
\end{lemma}

The proofs of Lemmas \ref{lem:interval} and \ref{lem:stick} will be deferred to later in this section. We will first examine how, invoking these lemmas, we can complete the proof of Theorem \ref{thm : lagrange}.

\begin{proof}[Proof of Theorem \ref{thm : lagrange}]
We assume without loss of generality that $a=a(\t)>0$, the case $a<0$ can be dealt with on similar lines.
Then from Lemma \ref{lem:interval}, we obtain that the optimal Lagrange parameter $\hat{\lambda}$ lies in the interval $(0,1/|h_{\max}|)$, and the function $g$, as defined in \eqref{eq:g}, is monotonically decreasing in this interval.


%
%


From Lemma \ref{lem:stick}, we obtain that, for any fixed $\epsilon>0$ (independent of $n$), if
\[
0 < \lambda \le \frac{1-\epsilon}{|h_{\max}|},
\]
then $g(\lambda)>0$.

Since $g$ is decreasing on the interval $\left(0, 1/|h_{\max}| \right)$, the last statement would imply that, $\hat{\la}$, which is a solution of the equation $g(\la)=0$, is greater than  $ \frac{1-\epsilon}{|h_{\max}|}$. 



This implies that, for any $\eps>0$, we have
\[
\frac{(1-\epsilon)}{|h_{\max}|}\le\hat{\lambda}\le \frac{1}{|h_{\max}|},
\]
with high probability as $n \to \infty$.

In other words, for any $\eps>0$, we have
\[
|\hat{\lambda}\cdot |h_{\max}| - 1 | \le \eps
\]
with high probability as $n \to \infty$, and since by assumption (A2) we have $|h_{\max}|=\m(1+o_P(1))$, we may therefore deduce that
\[
\hat{\lambda}\m = 1 + o_P(1),
\]
as desired.
\end{proof}

We now move on to the proofs of Lemma \ref{lem:interval}. 

\begin{proof}[Proof of Lemma \ref{lem:interval}]
Recall from \eqref{eq:MLE} that the optimal weights $\hat{w}_i$ are given by $\frac{1}{n} \frac{1}{1+\hat{\la} h_i}$. Since the optimal weights are non-negative, this implies that $1+\hat{\la}h_i >0$ for each $i$. In the context of the discussion above, we can apply this in particular to $h_{\max}=-|h_{\max}|$, and obtain $1-\hat{\la}|h_{\max}|>0$, which implies that $\hat{\la}<\frac{1}{|h_{\max}|}$. If \[\tilde{h}_{\max}:=\max\{{h_i : h_i >0}\},\] then, using the fact that $1+\hat{\la}\tilde{h}_{\max}>0$, we can deduce that $\hat{\la}> - 1\big/\tilde{h}_{\max}$. Notice that $\tilde{h}_{\max}>0$, so the above argument locates the optimal value $\hat{\la}$ of the Lagrangian parameter in the interval $\left( - 1\big/\tilde{h}_{\max} , 1/|h_{\max}| \right)$.

In the interval $\left( - 1\big/\tilde{h}_{\max} , 1/|h_{\max}| \right)$, the function $g(\la)$ is continuous. The fact that 
\[g^{\prime}(\lambda)=-\sum^n_{i=1}h^2_i/(1+\lambda h_i)^2<0 \text{ for all } \lambda \in \left( - 1\big/\tilde{h}_{\max} , 1/|h_{\max}| \right)\] 
implies that $g$ is, in fact, a decreasing function on this interval. As such, $g(\la)=0$ has a unique root in the interval $\left( - 1\big/\tilde{h}_{\max} , 1/|h_{\max}| \right)$, which must be the optimal Lagrange parameter $\hat{\la}$. Observe that, from \eqref{eq:functiong}, it follows that for large enough $n$, we have $g(0)>0$ with high probability, which implies that, in fact, $\hat{\la}>0$.
\end{proof}

We now proceed to the proof of Lemma \ref{lem:stick}.

\begin{proof}[Proof of Lemma \ref{lem:stick}]


Set $\lambda_0=(1-\epsilon)/|h_{\max}|$. Since, by Lemma \ref{lem:interval}, the function $g$ is decreasing on the interval $(0,1/|h_{\max}|)$, to prove the present lemma it suffices to show that $g(\lambda_0)>0$.

We begin with the identity
\begin{equation}\label{eq:lamb}
1+\lambda_0h_i=1-\frac{(1-\epsilon)}{h_{\max}}h_i=\left(1-\frac{h_i}{h_{\max}}\right)+\epsilon\frac{h_i}{h_{\max}.}
\end{equation}

We will decompose the index set $\mathcal{I}=\{1,\ldots,n\}$ as \[\mathcal{I}=\iz^\c \cup (\ip \cap \iz) \cup (\ic \cap \iz).\] 
We will accordingly consider three cases in order to obtain bounds on $\frac{h_i}{1+\lambda_0h_i}$.

\noindent{\bf Case $1$:} Suppose $i\in \iz^\c$.  For any $i\in\iz^\c$, by definition of $\iz^\c$ (see \eqref{eq:iz}), we have $|h_i/h_{\max}|\le \m^{-\delta}(1+o_P(1))$. It may be emphasised that this $o_P(1)$ is uniform in the index $i$, since it originates from the asymptotics of $h_{\max}$ as in assumption (A2). Thus, for any $i\in\iz^\c$, using \eqref{eq:lamb} we have
\[
1-\m^{-\delta} - o_P(1) \le 1+\lambda_0h_i\le 1+\m^{-\delta} + o_P(1) .
\]

So for all $i\in\iz^\c$, the relation
\begin{equation} \label{eq:i0bound}
(1-\sgn(h_i)\m^{-\delta} - o_P(1))h_i\le \frac{h_i}{1+\lambda_0h_i}\le (1 + \sgn(h_i)\m^{-\delta} + o_P(1) )h_i
\end{equation}
holds, with the $o_P(1)$ term being uniform in $i$.

\noindent{\bf Case $2$:} Suppose $i \in \ip \cap \iz$.

For $i \in \ip \cap \iz$, we observe that $\frac{h_i}{1+ \la_0 h_i} >0$, because $\la_0>0$.

\noindent{\bf Case $3$:} Finally, suppose $i \in \ic \cap \iz$. For all $i\in \ic\cap\iz$, we note that, by definition of $h_{\max}$, we have $|h_i/h_{\max}|\le 1$. Consequently, from \eqref{eq:lamb}, we deduce that
\[
1+\lambda_0h_i\ge \epsilon\frac{h_i}{h_{\max}} = \epsilon \frac{|h_i|}{|h_{\max}|},
\]
where, in the last step, we have used the fact that $h_i<0$ whenever $i \in \ic$.
The above arguments imply that
\[
\l|\frac{h_i}{1+\la_0 h_i} \r| \le \frac{\mid h_i\mid}{\epsilon\frac{\mid h_i\mid}{\mid h_{\max}\mid}}=\frac{\mid h_{\max}\mid}{\epsilon}.
\]


This completes the analysis of the three cases to obtain a lower bound on $1 + \lambda_0 h_i$. We can summarize the bounds as :
\begin{equation*}
\frac{h_i}{1 + \la_0 h_i} \ge \begin{cases}
(1- \sgn(h_i)\m^{-\delta} - o_P(1))h_i &\text{if $i \in \iz^\c$},\\
0 &\text{if $i \in \ip \cap \iz$}, \\
-\frac{|h_{\max}|}{\eps}  &\text{if $i \in \ic \cap \iz$}.
\end{cases}
\end{equation*}

We may rewrite the lower bound for $\iz^\c$ as 
\begin{equation} \label{eq:ri0}
\frac{h_i}{1 + \la_0 h_i} \ge h_i - \sgn(h_i)\m^{-\delta} h_i - o_P(1) \cdot h_i.
\end{equation} 
Putting 
\begin{equation} \label{eq:ri}
r_i:=  \sgn(h_i)\m^{-\delta} h_i - o_P(1) \cdot h_i
\end{equation}
as in the lower bound \eqref{eq:ri0}, we set 
\begin{equation} \label{eq:R}
\mathcal{R}=\frac{1}{n}\sum_{i \in \iz^\c} r_i.
\end{equation}
Proposition \ref{prop:tail} establishes that the random variable $\mathcal{R}$ is, in fact, $o_P(1)$.

Now it follows that, using the analysis of Cases 1-3, we have
\begin{align}
\frac{g(\lambda_0)}{n}=&\frac{1}{n}\sum_{i=1}\frac{h_i}{1+\lambda_0h_i} \nonumber\\
= &\frac{1}{n}\sum_{i \in \iz^\c}\frac{h_i}{1+\lambda_0h_i} + \frac{1}{n}\sum_{i \in \ip \cap \iz}\frac{h_i}{1+\lambda_0h_i} + \frac{1}{n}\sum_{i \in \ic \cap \iz}\frac{h_i}{1+\lambda_0h_i} \nonumber\\
\ge &\frac{1}{n}\sum_{i \in \iz^\c}\frac{h_i}{1+\lambda_0h_i} + \frac{1}{n}\sum_{i \in \ic \cap \iz}\frac{h_i}{1+\lambda_0h_i} \nonumber\\
\ge &\frac{1}{n}\left(\sum_{i\in\iz^\c}h_i \l(1-\sgn(h_i)\m^{-\del} - o_P(1)\r)\r)-\frac{\mid h_{\max}\mid}{\epsilon}\frac{\mid \ic \cap \iz\mid}{n}\nonumber\\
\ge &\frac{1}{n}\left(\sum_{i\in\iz^\c}h_i\right) - \mathcal{R} -\frac{\mid h_{\max}\mid}{\epsilon}\frac{\mid \iz\mid}{n}\nonumber\\
\ge &a -o_P(1), \label{eq:l}
\end{align}
where, in the last step, we invoke Proposition \ref{prop:tail} and Remark \ref{rem:i0order}.

Thus, we have $g(\lambda_0) \ge n(a -o_P(1))>0$ with high probability, since $a>0$. This completes the proof. 

\end{proof}

We end this section with Proposition \ref{prop:tail} and its proof.

\begin{proposition} \label{prop:tail}
Suppose $a=a(\t)>0$. Then we have 
\begin{itemize}
\item (i) $\frac{1}{n}\left(\sum_{i\in\iz^\c}h_i\right) = a - o_P(1)$.
\item(ii) $\mathcal{R}=o_P(1)$, where $\mathcal{R}$ is as defined in \eqref{eq:ri} and \eqref{eq:R}.
\end{itemize}
\end{proposition}

\begin{proof}
For $\delta$ small enough such that the assumption (A4) is valid, we  observe that 
\[
\frac{1}{n}\sum_{i\in\iz^\c}h_i=\frac{1}{n}\sum^n_{i=1}h_i-\frac{1}{n}\sum_{i\in\iz}h_i.
\] But
\[
\frac{1}{n}\left|\sum_{i\in\iz}h_i\right|\le\mid h_{\max}\mid\frac{\mid\iz\mid}{n}=(\m+o_P(1))O_P(p_{n,\del})=o_P(1),
\]
where, in the last step, we have used (A4). On the other hand, by the law of large numbers, $\frac{1}{n}\sum^n_{i=1}h_i = a +o_P(1)$. This completes the proof of part (i) of the proposition. 

For part (ii), we proceed as 
\begin{align*}
|\mathcal{R}| = & \l|\frac{1}{n} \sum_{i \in \iz^\c}  \l( \sgn(h_i)\m^{-\delta} h_i - o_P(1) \cdot h_i    \r) \r|  \\
 &\le  \m^{-\delta} \cdot \frac{1}{n} \sum_{i=1}^n   |h_i| + o_P(1) \cdot \frac{1}{n} \sum_{i=1}^n  |h_i| ,   \\
\end{align*}
where, in the last step, we have used the fact that the $o_P(1)$ term from \eqref{eq:i0bound} is uniform in $i$.  By the law of large numbers, $ \frac{1}{n} \sum_{i=1}^n   |h_i|  = \E[|h_1|] + o_P(1)$, which implies that $\mathcal{R}=o_P(1)$, as desired.


\end{proof}

\section{A canonical expansion for $\hl$}
In this section, we rigorously establish the canonical expansion \eqref{eq : expansion}.  We work with a fixed $\t \in \Theta$, and abbreviate $a(\t)$ and $\xi_i(\t)$ simply as $a$ and $\xi_i$ respectively.
\begin{proof}[Proof of Theorem \ref{thm : expansion}]
We will work in the setting  $a(\t)>0$; the case $a(\t)<0$ will follow along similar lines.

When $a=a(\t)>0$, we have  $h_{\max}<0$ and from Theorem \ref{thm : lagrange}, we obtain  $\hl=-\frac{1}{h_{\max}}+\zeta = \frac{1}{|h_{\max}|}+\zeta$, for some $\zeta\le 0$.

We further recall that $|h_{\max}|=\m(1+o_P(1))$ (assumption (A2) on rates of growth of maxima) and $\mid \zeta\mid=o_P\left(\frac{1}{\m}\right)$ (from Theorem \ref{thm : lagrange}). Together, these imply that $\hl=-\frac{1}{\m}+o_P(\frac{1}{\m})$.

Without loss of generality, let $h_1=h_{\max}$. Then we have
\[
\frac{h_1}{1+\hat{\la} h_1}=\frac{h_1}{1+\left(-\frac{1}{h_{\max}}+\zeta\right)h_{\max}}=\frac{1}{\zeta}.
\]
So we get,
\[
\sum^n_{i=1}\frac{h_i}{1+\hat{\la} h_i}=\frac{1}{\zeta}+\sum^n_{i=2}\frac{h_i}{1+\hat{\la} h_i}=0
\]
That is:
\begin{equation}\label{eq:x}
\frac{1}{\zeta}=-\sum^n_{i=2}\frac{h_i}{1+\hat{\la} h_i}
\end{equation}

Next investigate the R.H.S. of \eqref{eq:x}. To this end, we will express it as 
\begin{equation} \label{eq:x-expansion}
\sum^n_{i=2}\frac{h_i}{1+\hat{\la} h_i} = \l( \sum_{\{2,\ldots,n\}\cap\iz^\c} \frac{h_i}{1+\hat{\la} h_i} \r) + \l( \sum_{\{2,\ldots,n\}\cap \ip \cap \iz} \frac{h_i}{1+\hat{\la} h_i} \r) +  \l( \sum_{\{2,\ldots,n\}\cap \ic \cap \iz} \frac{h_i}{1+\hat{\la} h_i}  \r),
\end{equation}
where the sets $\iz,\ip$ and $\ic$ are as in \eqref{eq:iz}, \eqref{eq:ip} and \eqref{eq:ic} respectively. We will deal with these three summands separately, respectively in Propositions \ref{prop:izc}, \ref{prop:ip} and \ref{prop:ic}.  

Combining \eqref{eq:x} and \eqref{eq:x-expansion} with Propositions \ref{prop:izc}, \ref{prop:ip} and \ref{prop:ic}, we obtain
\begin{align}
\frac{1}{\zeta}&=-n(a+o_P(1)),\nonumber\\
\text{so that }\zeta&=-\frac{a^{-1}}{n}+o_P\left(\frac{1}{n}\right),\nonumber
\end{align}
which completes the proof of Theorem \ref{thm : expansion}.
\end{proof}

We now move on to first to the statements and subsequently to the proofs of Propositions \ref{prop:izc}, \ref{prop:ip} and \ref{prop:ic}.
We begin with
\begin{proposition} \label{prop:izc}
Let $a(\t)=a>0$ and $h_{\max}=h_1$. Then we have,
\[\frac{1}{n}\sum_{\{2,\ldots,n\}\cap\iz^\c} \frac{h_i}{1+\hat{\la} h_i} = a +o_P(1).\]
\end{proposition}

We continue with 
\begin{proposition} \label{prop:ip}
Let $a(\t)=a>0$ and $h_{\max}=h_1$. Then we have,
\[\frac{1}{n}\sum_{\{2,\ldots,n\}\cap\ip\cap\iz} \frac{h_i}{1+\hat{\la} h_i} = o_P(1).\]
\end{proposition}

Finally, we state
\begin{proposition} \label{prop:ic}
Let $a(\t)=a>0$ and $h_{\max}=h_1$. Then we have,
\[\frac{1}{n}\sum_{\{2,\ldots,n\}\cap\ic\cap\iz} \frac{h_i}{1+\hat{\la} h_i} = o_P(1).\]
\end{proposition}

We now provide the proofs of Propositions \ref{prop:izc} through \ref{prop:ic} in succession.

\begin{proof}[Proof of Proposition \ref{prop:izc}]
For $i\in\iz^\c$, we have $|h_i|=O\left(\m^{1-\delta}\right)$ by definition of $\iz$. This implies that
\[
\frac{h_i}{1+\hat{\la} h_i}=\frac{h_i}{1+\left(-\frac{1}{\m}+ o_P(1/\m)\right)h_i}=\frac{h_i}{1+o_P\left(\m^{-\del}\right)}=h_i\bigg(1+o_P\left(\m^{-\del}\right)\bigg).
\]
We then observe that
\begin{align}
\frac{1}{n}\sum_{\{2,\ldots,n\}\cap\iz^\c}\frac{h_i}{1+\hat{\la} h_i}=&\left\{1+o_P\left(\m^{-\del}\right)\right\} \cdot \frac{1}{n}\left( \sum_{\{2,\ldots,n\}\cap\iz^\c}h_i \right). 
 \label{eq:122terms}
\end{align}
From the assumption (A2) on the growth rate of maxima and the definition of $\iz$, it follows that $h_{\max} \notin \iz^\c$.  Recall that in our case $h_{\max}=h_1$. As a result, the R.H.S. of \eqref{eq:122terms} can be written as \[ \left\{1+O_P\left(\m^{-\del}\right)\right\} \cdot \frac{1}{n}\left( \sum_{i \in \iz^\c}h_i \right). \] 

But then it follows from Proposition \ref{prop:tail}, part (i), that \[\frac{1}{n}\left( \sum_{i \in \iz^\c}h_i \right)=a - o_P(1),\] whence the R.H.S. of \eqref{eq:122terms}  reduces to $a + o_P(1)$, thereby completing the proof of the proposition.
\end{proof}

We continue our analysis with

\begin{proof}[Proof of Proposition \ref{prop:ip}]
For $i\in \ip \cap \iz$, we have $1+\hat{\la} h_i>1$  since $\hat{\la}>0$ and $h_i>0$. Consequently, we have
\begin{equation} \label{eq:case2}\mid h_i/(1+\hat{\la} h_i)\mid\le h_i \le \m(1+o_P(1)). \end{equation}
As a result, we can write
\begin{align}
 & \left|\frac{1}{n}\sum_{\{2,\ldots,n\}\cap \ip \cap \iz}\frac{h_i}{1+\hat{\la} h_i}\right|  \nonumber \\
\le & \left(\frac{1}{n}\sum_{\{2,\ldots,n\}\cap \ip \cap \iz} h_i \right) \nonumber \\
\le & \m (1+o_P(1)) \cdot \frac{\mid \ip \cap \iz \mid}{n}   \nonumber \\
\le &\m\cdot \frac{\mid \iz \mid}{n} \cdot (1+o_P(1)).   \label{eq:f}
\end{align}
Since $\E[\mid \iz\mid/n] \le  p_{n,\del}$ and $\m p_{n,\del}=o_P(1)$ as per assumption (A4), we have $\m \cdot \frac{\mid \iz \mid}{n} =o_P(1)$ as $n \to \infty$. This completes the proof of the proposition.
\end{proof}

We end this section with

\begin{proof}[Proof of Proposition \ref{prop:ip}]
For $i \in \ic \cap \iz$, we write:
\[
\left|\frac{h_i}{1+\hat{\la} h_i}\right|=\left|\frac{h_i}{1+\left(-\frac{1}{h_{\max}}+\zeta\right)h_i}\right|=\left|\frac{h_i}{\frac{h_{\max}-h_i}{h_{\max}}+ \zeta h_i}\right|.
\]
Now recall that $\zeta\le 0$, $h_{\max}<0$ and $h_i<0$ (the last assertion being true since $i \in \ic$). These, in particular, imply that $\zeta h_i\ge 0$ and $(h_{\max}-h_i)/h_{\max}\ge 0$.

Consequently, we have
\[
\left|\frac{h_i}{\frac{h_{\max}-h_i}{h_{\max}}+\zeta h_i}\right|=\frac{|h_i|}{\frac{h_{\max}-h_i}{h_{\max}}+\zeta h_i}\le\frac{|h_i|h_{\max}}{h_{\max}-h_i}.
\]
Now, by assumption (A2) we have $\mid h_i\mid \le \m(1+o_P(1))$ for all $i$, and the same inequality holds for $h_{\max}$. On the other hand,  by assumption (A3) we have \[\mid h_{\max}-h_i\mid \ge \mid \xi^{-}_{\max}- \xi^{-}_{\max,2}\mid \ge \m^{-\gamma}(1+o_P(1)).\]
This implies that, for $i \in \ic \cap \iz$ we have
\begin{equation} \label{eq:case3}
\left|\frac{h_i}{1+\hat{\la} h_i}\right|  \le \l| \frac{|h_i|h_{\max}}{h_{\max}-h_i}  \r|  \le \m^{\gamma+2}(1+o_P(1)).
\end{equation}
Then we have
\begin{align}
 &\left|\frac{1}{n}\sum_{\{2,\ldots,n\}\cap \ic \cap \iz}\frac{h_i}{1+\hat{\la} h_i}\right| \nonumber \\
\le &\frac{1}{n}\sum_{\{2,\ldots,n\}\cap \ic \cap \iz} \left| \frac{h_i}{1+\hat{\la} h_i}\right| \nonumber \\
\le & \m^{\gamma+2} (1+o_P(1)) \cdot  \frac{\mid \ic \cap \iz \mid}{n}  \nonumber \\
\le &\m^{\gamma+2} \cdot \frac{\mid \iz \mid}{n} \cdot (1+o_P(1)).   \label{eq:f}
\end{align}
Since $\E[\mid \iz\mid/n] \le  p_{n,\del}$ and $\m^{\gamma+2} p_{n,\del}=o_P(1)$ as per assumption (A4), we have $\m^{\gamma+2} \cdot \frac{\mid \iz \mid}{n} =o_P(1)$ as $n \to \infty$. This completes the proof of the proposition.
\end{proof}

\section{Degeneracies in the MLE measure}
In this section, we establish the existence of degeneracies in the MLE measure under mis-specification, and contrast it with the case under correct specification where such degeneracies are not present.
 We work with a fixed $\t \in \Theta$, and abbreviate $a(\t)$ and $\xi_i(\t)$ simply as $a$ and $\xi_i$ respectively.
\begin{proof}[Proof of Theorem \ref{thm : degeneracy}]
We will work in the setting $a=a(\t)>0$; the case $a(\t)<0$ will follow on similar lines.
For each $i=1$, $2$, $\ldots$, $n$, we have
\begin{equation} \label{eq:wtformula}
\hat{w}_i=\frac{1}{n}\frac{1}{1+\hat{\la} h_i}.
\end{equation}
Using this expression, we may deduce that
\begin{align}
\hat{w}_{max}= & \frac{1}{n} \cdot \frac{1}{1+\left(-\frac{1}{h_{\max}}+\zeta\right) h_{\max}} \nonumber \\ = &\frac{1}{n\zeta h_{\max}} \nonumber \\
= &\frac{a}{-h_{\max}}(1+o_P(1)) \text{  [using Theorem \ref{thm : expansion}]} \nonumber \\
 = & \frac{a}{|h_{\max}|}(1+o_P(1)) \nonumber \\ =  & a\cdot \m^{-1}(1+o_P(1)) \nonumber.
\end{align}

Denote by $i_{\max}$ the index $i$ for which $h_i=h_{\max}$. Recall that $\hat{\la}>0$.
For $i\ne i_{max}$ such that $h_i>0$, we have, from \eqref{eq:wtformula}, the inequality $\hat{w}_i<1/n$ .  For $i \ne i_{\max}$ s.t. $h_i<0$, we proceed as
\begin{align}
\hat{w}_i=\frac{1}{n} \cdot \frac{1}{1+\left(-\frac{1}{h_{\max}}+ \zeta \right) h_i}=\frac{1}{n} \cdot \frac{1}{1-\frac{h_i}{h_{\max}}+\zeta h_i}.
\end{align}
Recall from Theorem \ref{thm : expansion} that, since we are in the setting $a>0$, we have $\zeta<0$ as well. Thus,
\[
\hat{w}_i\le\frac{1}{n}\frac{1}{1-\frac{h_i}{h_{\max}}}\le \frac{1}{n} \cdot  \frac{h_{\max}}{h_{\max} - h_i}  \le  \frac{\m^{\gamma+1}(1+o_P(1))}{n},
\]
where, in the last step, we have used assumption (A4).

Combining the above analyses to cover all $i \ne i_{\max}$, we deduce that
\[
\max\{\hat{w}_i~:~i\ne i_{max}\}=O\left(\frac{\m^{\gamma+1}}{n}\right).
\]

On the other hand, we observe that \[|\hl h_i|=|\hl| |h_i| \le \m^{-1}(1+o_P(1))\cdot \m(1+o_p(1))=1+o_P(1),\] where, in the last step, we have used Theorem \ref{thm : degeneracy} and assumption (A2).
Applying this to \eqref{eq:wtformula}, we deduce that \[ \min\{\hat{w}_i~:~1\le i \le n\}\ge \frac{1}{n}(1-o_P(1)), \] as desired.

\end{proof}

\section{Asymptotics of the Wilks' statistic}
In this section, we demonstrate the anomalous behaviour of the Wilks' statistic for empirical likelihood in the misspecified setting, and in doing so, establish Theorem \ref{thm : Wilks} and Corollary \ref{cor:Wilks}.
We will work with a fixed $\t \in \Theta$, and abbreviate $a(\t)$ and $\xi_i(\t)$ simply as $a$ and $\xi_i$ respectively.
\begin{proof}[Proof of Theorem \ref{thm : Wilks}]
As in the proof of Theorem \ref{thm : expansion}, we will work in the setting $a=a(\t)>0$, the case when $a(\t)$ is negative will follow on similar lines.
We begin with
\[ -2\log n \h{w}_i  = 2 \log (1+ \hl h_i). \]
Recall the parameter $\del$ from (A4).
We divide the indices $i$ into two groups : $\I_0$ such that $|h_i|>\m^{1-\del}$ for $i \in \I_0$, and $\I_0^\c$ consisting of the rest of the indices.

This leads to the expression
\begin{equation} \label{eq:Wsplit}
\mathbf{L} = \left(\sum_{i \in \I_0^\c} -2\log n \h{w}_i \right) + \left(\sum_{i \in \I_0} -2\log n \h{w}_i \right)
\end{equation}

For $i \in \I_0^\c$, we deduce from Theorem \ref{thm : expansion} that $|\hl h_i| \le \m^{-\del}(1+o_P(1))$, so for $n$ large enough (such that $\m^{-\del}<1/2$) we can expand
\begin{equation} \label{eq:logexp} 2\log(1+\hl h_i)=2\hl h_i + \Xi_i, \end{equation}
where $ |\Xi_i| \le 4  |\hl|^2 h_i^2$ (with probability tending to 1 as $n \to \infty$). Thus, the contribution to $\mathbf{L}$ due to the indices in $\I_0^\c$ can be written as
\begin{equation} \label{eq:contr-good}
\left(\sum_{i \in \I_0^\c} -2\log n \h{w}_i \right) = 2 \hl \sum_{i \in \I_0^\c} h_i + \Xi,
\end{equation}
where $\Xi=\sum_{i \in \I_0^\c}  \Xi_i$ satisfies
\begin{equation} \label{eq:exp-tail}
|\Xi| \le 4 |\hl|^2 \sum_{i \in I_0^\c} h_i^2 \le 4 |\hl|^2 \sum_{i=1}^n h_i^2 = 4 n |\hl|^2 \frac{1}{n} \sum_{i=1}^n h_i^2 =  \frac{4n}{\m^2} \E[h_1^2] (1+o_P(1)),
\end{equation}
where in the last step we have used Theorem \ref{thm : expansion} and the Law of Large Numbers for $\frac{1}{n}\sum_{i=1}^n h_i^2$.
Also observe, using Proposition \ref{prop:tail}, that \[ 2 \hl \sum_{i \in \I_0^\c} h_i  = 2 \hl (a-o_P(1)) .  \] Combined with Theorem  \ref{thm : expansion}., this yields
\begin{equation} \label{eq:exp-leadorder}
2 \hl \sum_{i \in \I_0^\c} h_i  = 2n \hl \cdot \l(a+o_P(1) \r) = \frac{2n}{\m} (a+o_P(1)),
\end{equation}
 In view of \eqref{eq:contr-good}, this implies
\begin{equation} \label{eq:wgoodterms}
\left(\sum_{i \in \I_0^\c} -2\log n \h{w}_i \right) = \frac{2n}{\m} (a(\t)+o_P(1))
\end{equation}

From Theorem \ref{thm : degeneracy}, in particular \eqref{eq : deg3}, we may deduce that, for all $i$ we have $1 \ge \h{w}_i \ge \frac{1}{n}(1-o_P(1))$.
This implies, in particular, that for any $i$ we have \[ \mid \log n\hat{w}_i \mid \le \log n. \]
In view of this, we can bound the contribution to $\mathbf{L}$ from the indices in $\I_0$ as
\[ |\sum_{i \in \I_0} -2\log n \h{w}_i | \le 2 | \I_0 | \log n.   \]
 Recall that $\E[|\I_0|] = n p_{n,\del}$, so we have
 \begin{align}
 &\E[ |\sum_{i \in \I_0} -2\log n \h{w}_i | ] \le  2 \log n \cdot \E[|\I_0|] \nonumber \\ \le &\quad 2 \log n \cdot n p_{n,\del}
 = \frac{2n}{\m} \cdot \m \log n \cdot p_{n,\del} \nonumber \\ = &\quad \frac{2n}{\m} \cdot o_P(1),
 \end{align}
 where in the last step we have invoked assumption (A4). This, in particular, implies that
 \begin{equation} \label{eq:wbadterms}
\mid \sum_{i \in \I_0} -2\log n \h{w}_i \mid = \frac{2n}{\m} \cdot o_P(1).
 \end{equation}

Combining \eqref{eq:Wsplit}, \eqref{eq:wgoodterms} and \eqref{eq:wbadterms}, we deduce that \[\mathbf{L}=2\cdot \frac{n}{\m} \cdot a(\t)(1+o_P(1)),\] as desired.
\end{proof}

\begin{proof}[Proof of Corollary \ref{cor:Wilks}]
The proof of Corollary \ref{cor:Wilks} follows on similar lines to the proof of Theorem \ref{thm : expansion}, by taking the expansion \eqref{eq:logexp} to degree $k$ (instead of terminating at degree 1).
\end{proof}


\section{The Bayesian setting and posterior consistency}
\label{sec : consistent}
In this section, we discuss the convergence of the (random) posterior distribution from a BayesEL procedure to the delta measure $\delta_{\t_0}$ where $\t_0$ is the \textit{true value} of the parameter $\t$ (i.e., the one from which the data is generated).  This will culminate in the proof of Theorem \ref{thm : consistent}, but first we need to lay down a set-up  and a class of assumptions under which the main result of this section will hold.

We work in the setting of :
\begin{itemize}
\item (B1) There is a compact parameter space $\Theta \subset \R^d$ with open interior, and containing the true parameter value $\t_0$ in the interior.
\item (B2) There is a collection of i.i.d. random variables $X_i(\t)$ (taking values in some space $E$), jointly defined on the parameter space $\Theta$ as random fields. Observe that, in the context of the notations in Section \ref{sec:notpd}, $X_i(\t)$ can in particular be taken to be the tuple $(X_i,\t)$, so that the set up discussed in Section \ref{sec:notpd} are covered in the framework of the present section.
\item (B3) There is a prior distribution on $\Theta$ that has a density $\pi$ (with respect to the Lebesgue measure on $\Theta$) that is positive and lower semi-continuous at $\t_0$ (which implies that $\pi(\t)$ is uniformly bounded away from 0 in a neighbourhood of $\t_0$).
\item (B4) There is a measurable function $h : E \to \R$ such that $h(X_i(\t))$ are the observed variables (with finite second moments are uniformly bounded in $\t \in \Theta$), and the following laws of large numbers hold uniformly for $\pi$-a.e. $\t$ : \[\frac{\sum_{i=1}^n h(X_i(\t))}{n} \to \E[h(X_1(\t))]~~ \mathrm{a.s.}  \] and
\[ \frac{\sum_{i=1}^n h(X_i(\t))^2}{n} \to \E[h(X_1(\t))^2]~~\mathrm{a.s.} .  \]
Define the random variables \[ r_{n,1} = \mathrm{ess sup}_\t \l| \frac{\sum_{i=1}^n h(X_i(\t))}{n} - \E[h(X_1(\t))]   \r|  \] and
\[ r_{n,2} =  \mathrm{ess sup}_\t \l| \frac{\sum_{i=1}^n h(X_i(\t))^2}{n} - \E[h(X_1(\t))^2]   \r| .  \]
Then $\pi$-a.e. uniform convergence implies that $r_{n,1}$ and $r_{n,2} \to 0$ with high probability. In particular, we assume that $\max\{r_{n,1},r_{n,2}\}  = o_P(n^{-\del})$ for some $\del>0$.
\item (B5) The expectation function $\g(\t)=\E[h(X_i(\t))]<\infty$ \textit{strongly identifies} $\t_0$, that is $\g$ satisfies $\g(\t_0)=0$ and $\g(\t) \ne 0$ for $\t \ne \t_0$ in the strong sense - in other words, $\inf_{\t \in B^\c }|\g(\t)| > 0$ for any neighbourhood $B$ of $\t_0$.  Moreover, we assume that $\g$ is 1-Lipschitz at $\t_0$ (i.e. $|\g(\t)| \le c |\t - \t_0|$ on some neighbourhood of $\t_0$). Notice that, this implies $\g$ is bounded on a small enough neighbourhood of $\t_0$.
\item (B6) There exists a deterministic sequence $\M$, possibly depending on $\t$, such that  $\max_{1\le i \le n}|\hi|=\M(1+o_P(1))$. Further, let \newline $M_n=\sup_{\t \in \Theta} \M$ and $m_n=\inf_{\t \in \Theta} \M$. We assume that $m_n \to \infty$ a.s. and $M_n=o_P(n)$ a.s. Furthermore, we assume that $M_n/m_n =o_P(n^{\del/2})$, where $\del$ is as in (B4).
  \item (B7) We assume that, uniformly in $\t\in\Theta$, $\mid\hl(\t)\mid=O_P(1/\M)$ holds.
\end{itemize}

Note that, the strong identifiability of $\g$ in (B5) can be deduced as a consequence of a simpler but weaker condition that $\g$ identifies $\t_0$ (i.e. $\g(\t)=0$ iff $\t=\t_0$) and $\g$ is continuous.  For many models $M_n$ and $m_n$ asymptotically grow at the same rate with $n$.  That is, the assumption (B6) would be easily satisfied.

  Since for each $i=1$, $2$, $\ldots$, $n$, the inequality $(1+\hl(\t) h(X_i(\t)))\ge 0$ holds (see, e.g. \eqref{eq:MLE}), we  have $\mid\hl(\t)\mid\le 1/\M$ with high probability, provided the largest in magnitude among the positive and the negative $h_i$-s both grow at the rate $\M$. This is true under very general conditions, e.g. in the setting of assumption (A2) earlier in this paper.
  This would imply that the assumption (B7) would be satisfied.
  
A general location family of distributions (with $X_i(\t)=X_i+\t$) would satisfy all these assumptions.  With few modifications, the proof below extends to any dimension, provided an appropriate bound for $\|\hl(\t)\|$ along the lines of (B7) can be found. Such bounds can be shown to hold under general conditions already considered in the literature, see e.g. \citep{chaudhuri2018easy}, in particular, conditions (A1)-(A3) therein.

\begin{proof}[Proof of Theorem \ref{thm : consistent}]
We recast the maximum log likelihood $L_n(\t)$ as follows:
\begin{equation} \label{eq : repr} L_n(\t) = \min_\la \{  -\sum_{i=1}^n \log (1+\la h(X_i(\t)) \} \text{ such that } 1 + \la h(X_i(\t)) >0 \forall i=1,\cdots,n.  
 \end{equation}
This representation is related to the fact that the optimal weights for the MLE in empirical likelihood are given by 
\begin{equation} \label{eq:MLEBayes}
\hw_i=\frac{1}{n} \frac{1}{1+\hl h_i},
\end{equation}
see \eqref{eq:MLE};  whereas the maximum log likelihood is simply $\sum_{i=1}^n \log \hw_i$, see \eqref{eq:emplik}. The constraints $1 + \la h(X_i(\t)) >0 \forall i$ are related to the fact that the optimal weights should satisfy $0 \le \hw_i \le 1$, and they are given by \eqref{eq:MLEBayes}. For details, we refer the reader to \citep[Section 3.14]{owenbook}.
The efficacy of this representation will become clear in the argument that follows.
We write the posterior distribution as the (random) measure
  \[ \Pi^{(n)}(\t) = \gamma_n(\t) \pi(\t) \d \t \bigg/ \int_\Theta \gamma_n(\t) \pi(\t) \d \t,  \] where $\gamma_n(\t) = \exp(L_n(\t))$.

Our proof, broadly speaking, will proceed along the following contour. In order to show that $\Pi^{(n)}(\t)  \to \del_{\t_0}$ weakly, it suffices  to show that $\int f(\t) \Pi^{(n)}(\t) \to f(\t_0)$ for all bounded continuous functions $f: \Theta \to  \R$. Equivalently, we will show that for any open ball $B \subset \R^d$ containing $\t_0$, we have $\int_{B^\c} f(\t) \d \Pi^{(n)}(\t) \to 0$.

To accomplish the latter, our argument will consist of two ingredients : an upper bound on   $\gamma_n(\t)$ (roughly giving exponential decay in $n$ outside $B$, upto logarithmic factors), and a lower bound on $ \int_\Theta \gamma_n(\t) \pi(\t) \d \t$ (giving a sub-exponential bound).

\subsection{The upper bound}
Let $B$ be a neighbourhood of $\t_0$. We consider $\t $ such that $\t \in B^\c$.  By strong identifiability it follows that $\g(\t)\ne0$.  Now suppose we consider $\la_1(\t)= \mathrm{Sign}[\g(\t)]/(100\cdot\M)$, where the choice of $100$ is totally arbitrary.  It follows from the definition of $\M$ that $|\la_1(\t) \cdot h(X_i(\t))| \le 1/100$ for all $i$, so in particular $\la_1(\t)$ is a candidate for the minimization problem \eqref{eq : repr}. As such, we obtain
\[ L(\t) \le  -\sum_{i=1}^n \log (1+\la_1(\t) h(X_i(\t)).   \]
Now, $|\la_1(\t) h(X_i(\t))|\le 1/100$ implies that, via an expansion of the logarithmic series, we have \[- \log (1+\la_1(\t) h(X_i(\t))) \le  - \la_1(\t) h(X_i(\t)) +   (\la_1(\t) h(X_i(\t)))^2. \]

Now from the definition of $L_n(\t)$ we can proceed as:
\begin{align}
\frac{1}{n}L_n(\t)&\le - \la_1(\t) \cdot \frac{1}{n} \l( \sum_{i=1}^n \hi  \r) + \la_1(\t)^2 \cdot \frac{1}{n} \l( \sum_{i=1}^n \hi^2  \r) \nonumber \\
& \le - \la_1(\t) \cdot \E[h(X_1(\t))] +  \la_1(\t) r_{n,1} +\la_1(\t)^2 \l( \E[h(X_1(\t))^2] + r_{n,2}  \r) \nonumber \\
& = - \la_1(\t) \cdot \g(\t)+  \la_1(\t) r_{n,1}  + \la_1(\t)^2\l( \E[h(X_1(\t))^2] + r_{n,2}  \r) \label{eq:B-ubound-expansion}
\end{align}
Now notice that, since $\M$ diverges as $n\rightarrow\infty$, $\la_1(\t) =O(1/\M) \rightarrow 0$.  Moreover,  by strong identifiability of $\g$ (see assumption (B5)),  $\g(\t)$ is uniformly bounded away from $0$ on $B^\c$.  Furthermore the $\pi$-a.e. uniform convergence assumption (B4) implies that $r_{n,1}$ and $r_{n,2} \to 0$ with high probability. Finally, by assumption (B2), the second moment $\E[h(X_1(\t))^2]$ is uniformly bounded in $\t \in B^\c$.  Combining these observations, we deduce that the $- \la_1(\t) \cdot \g(\t)$ term dominates in the upper bound in \eqref{eq:B-ubound-expansion}. Consequently, we have
\begin{align*}
 \frac{1}{n}L_n(\t)&\le - \la_1(\t) \cdot \g(\t) \l( 1+ o_P(1) \r) \text{   \hfill[using strong identifiability of $\g$]}\\
& = - |\g(\t)|/(100\cdot\M) \cdot \l(1+ o_P(1) \r) \text{   \hfill[substituting $\la_1(\t)$]} \numberthis \label{eq:B-ubound-1}
\end{align*}

Putting together all of the above, we deduce that on $B^\c$ we have
\begin{equation} \label{eq : ubound}
\gamma_n(\t) \le \exp\l( - \frac{n}{100\cdot\M} |\g(\t)|(1+o_P(1)) \r),
\end{equation}
completing the proof of the upper bound.

\subsection{The lower bound}
In this section, we show that, with high probability, $\int_\Theta \gamma_n(\t) \pi(\t) \d(\t) \ge \exp(-an^{1-\del/2}/m_n(1+o_P(1)))$ for some absolute constant $a>0$ and $\del$ as in (B4). We will work with $\t$ in a small enough neighbourhood $U$ of $\t_0$ such that $\g$ is bounded on that neighbourhood, further specifications on the precise choice of $U$ will be outlined later.

We start by observing that, by assumption (B7), $|\hl(\t)| =O(1/\M)$ and $(\g(\t) + r_{n,1})$ is $O_P(1)$ on $U$. Now observe that, for $\t$ in this neighbourhood $U$ we have, using Jensen's inequality,
\begin{equation} \label{eq:B-lbound-first}
  \frac{1}{n}L_n(\t)=- \frac{1}{n} \sum_{i=1}^n \log(1+\hl(\t) \hi) \ge - \log \l( 1 + \hl(\t) \cdot \frac{\sum_{i=1}^n \hi}{n} \r).
  \end{equation}
Since $\l| \frac{1}{n}\sum_{i=1}^n \hi \r| \le |\g(\t)| + r_{n,1}$, and since $-\log$ is a monotonically decreasing function, we can further lower bound the right hand side of \eqref{eq:B-lbound-first} as
  \begin{align*} 
 - \log \l( 1 + \hl(\t) \cdot \frac{\sum_{i=1}^n \hi}{n} \r) \ge &- \log \l( 1 + |\hl(\t)| \cdot \l| \frac{\sum_{i=1}^n \hi}{n}  \r| \r) \\ 
 \ge &- \log \l( 1 + |\hl(\t) | (|\g(\t)| + r_{n,1}) \r). \numberthis \label{eq:B-lbound-second}
\end{align*}
 Notice that $\mid\hl(\t)\mid = O_P(1/m_n)$ - consequence of assumption (B7) and the definition of $m_n$. Furthermore, with high probability, $(\g(\t) + r_{n,1})$ is uniformly bounded in $\t \in U$. As a result, $|\hl(\t)|(|\g(\t)|+r_{n,1}) =o_P(1)$. It mat be noted that the analysis in this sub-section until this point remains valid as long as the function $\g$ is known to be bounded on the set $\Theta$.
 
Via the lower bound $-\log (1+x) \ge -2x$ for small enough $x \ge 0$, together with \eqref{eq:B-lbound-second}, these imply that
\begin{align*}
\frac{1}{n}L_n(\t)& \ge -2 \mid\hl(\t)\mid ( \mid\g(\t)\mid +  r_{n,1})\\
& \ge -2c\mid\hl(\t)\mid |\t - \t_0| - 2 \mid\hl(\t)\mid \mid r_{n,1}\mid \numberthis \label{eq:B-lbound-third}
\end{align*}
where, in the last step, we have used the fact that $\g$ is 1-Lipschitz at $\t_0$ (assumption (B5), and $\g(\t_0)=0$) . 

 On the other hand, $|\hl(\t)|=O_P(1/\M)=o_P(n^{\del/2}/m_n)$, where $\del$ is as in assumption (B4) 
\begin{remark} \label{rem:lbound} 
 Note that we could have also taken, e.g.,  $\log \log n$ (or indeed, any sequence going to $\infty$) instead of $n^{\del/2}$ in the last $o_P$ bound; we work with $n^{\del/2}$ purely for the sake of notational simplicity. 
\end{remark} 
This, together with \eqref{eq:B-lbound-third}, implies that for $\t \in U$, we have with high probability
\[\gamma_n(\t) = \exp(n \cdot \frac{1}{n}L_n(\t)) \ge \exp(-2n^{1+\del/2} r_{n,1}/m_n) \exp(-2cn^{1+\del/2} |\t - \t_0|/m_n) \pi(\t)\] and therefore
\[\int_U\gamma_n(\t) \pi(t) \d t \ge \exp(-2n^{1+\del/2} r_{n,1}/m_n)\int_U \exp(-2cn^{1+\del/2} |\t - \t_0|/m_n) \pi(\t)\d \t.\]
Choose $U$ to be the set of all $\t$ such that $2cn^{1+\del/2}|\t-\t_0|/m_n \le 1$ and that $\pi(\t)>b$ for some $b>0$ for all $\t \in U$ (the  last condition being guaranteed by assumption (B3)). Using the fact that $r_{n,1}=o_P(n^{-\delta})$ (assumption (B4)), this implies that with high probability we have
$\gamma_n(\t)\ge \exp(-2n^{1-\del/2}/m_n)\cdot e^{-1}\cdot\pi(\t)$ on $U$. Also, recall that $U \subset \Theta \subset \R^d$. Recall that, by assumption (B6) we have $m_n \le M_n = o_P(n)$, which in particular implies that Vol($U$) is decaying polynomially in $n$.  Then we have, with high probability
\begin{align*}&  \int_\Theta \gamma_n(\t) \pi(\t) \d \t \ge  \int_U\gamma(\t) \pi(\t) \d \t  \ge \exp(-2n^{1-\del/2}/m_n)\cdot e^{-1}\cdot \mathrm{Vol}(U) \cdot b \\ &=  \exp(-2n^{1-\del/2}/m_n)\cdot e^{-1}\cdot (m_n/2cn^{1+\del/2})^d  \cdot b \ge \exp(-a n^{1-\del/2}/m_n) \end{align*} for some absolute constant $a>0$, as desired. 

\subsection{Combining the upper and the lower bounds}
We now combine the upper and the lower bounds obtained in the previous two sections.

Suppose $B$ is a neighbourhood of $\t_0$. Then, with high probability, we have
\begin{align*} &\int_{B^\c}\Pi^{(n)}(\t) \d \t = \l( \int_{B^\c} \gamma_n(\t) \pi(\t) \d \t \r) \bigg/ \l(\int_{\Theta} \gamma_n(\t) \pi(\t) \d \t \r) \\ \le & \int_{B^\c} \exp\l( - \frac{n}{100\cdot M_n} |\g(\t)|(1+o_P(1)) \r) \pi(\t) \d \t \bigg/  \exp(-a n^{1-\del/2}/m_n),
\end{align*}
with the $o_P(1)$ term being uniform in $\t \in B^\c$.
\textcolor{black}{By assumption (B5), there exists $\ell(B)>0$ such that $\mid\g(\t)\mid\ge 100\cdot\ell(B) \forall \t \in B^\c$. It follows that with high probability we have}
\begin{align*}
  \int_{B^\c}\Pi^{(n)}(\t) \d \t&\le \exp\l(\frac{a n^{1-\del/2}}{m_n}\r) \exp\l( - \frac{n}{100 M_n} \ell(B)(1+o_P(1)) \r)\\
  &=\exp\l(\frac{a n^{1-\del/2}}{m_n}-\frac{n}{M_n} \ell(B) (1+o_P(1))\r).  \numberthis \label{eq:rate1}  \end{align*}
Therefore, for any bounded continuous function $f:\Theta \to \R$, we have with high probability,
\begin{align*}
&\l| \int_\Theta f(\t)~\Pi^{(n)}(\t)~\d \t -f(\t_0)\r| = \l| \int_\Theta f(\t)~\Pi^{(n)}(\t)~\d \t - \int_\Theta f(\t_0)~\Pi^{(n)}(\t)~\d \t \r| \\
\le &\int_\Theta |f(\t) -f(\t_0) | ~\Pi^{(n)}(\t)~\d \t\\   \le&  \int_B |f(\t) -f(\t_0)| ~\Pi^{(n)}(\t)~\d \t + \int_{B^\c} |f(\t)-f(\t_0)| ~\Pi^{(n)}(\t)~\d \t.
\end{align*}
For any given $\eps>0$, by continuity of $f$ we can choose $B=B_\eps$ such that $|f(\t)-f(\t_0)| \le \eps$ for all $\t \in B_\eps$. Then
\[ \l| \int_\Theta f(\t) \Pi^{(n)} (\t) \d \t -f(\t_0)\r| \le \eps +  \exp\l(\frac{an^{1-\del/2}}{m_n}-\frac{n}{M_n} \ell(B) (1+o_P(1))\r). \]
Using $(B6)$, we see that, with high probability we have
\begin{equation} \label{eq:rate2}
\exp\l(\frac{an^{1-\del/2}}{m_n}-\frac{n}{M_n} \ell(B) (1+o_P(1))\r) = \exp\l( -\frac{n}{M_n} \ell(B)(1+o_P(1)) \r).
\end{equation}

Holding $\eps$ fixed and letting $n \to \infty$, we deduce that, with high probability  \[ \l| \int_\Theta f(\t) ~\Pi^{(n)}(\t)~\d \t -f(\t_0)\r|\le \eps  + o_P(1).\] But this implies that $\int_\Theta f(\t) \Pi^{(n)}(\t)  \d  (\t) \to f(\t_0)$ in probability as $n \to \infty$. 

Since this is true for any bounded continuous function $f$, we deduce that $\Pi^{(n)}(\t) \to \del_{\t_0}$ in probability, as $n \to \infty$. It may be noted that, combining \eqref{eq:rate1} and \eqref{eq:rate2}, we obtain a decay rate of $\exp\l( -\frac{n}{M_n} \ell(B) \r)$ for $\Pi^{(n)}(B^\c)$. 
\end{proof}

\begin{remark} \label{rem:Bayesian-1}
It may be observed that for a location family, the quantity $\ell(B)$ can be taken to be a measure of the ``deviation from the truth'', that is, the radius of the ball $B$. Accordingly, if $\t_0 \in B$, then the posterior measure of $B^\c$ converges to 0  at a rate $\le \exp(-\frac{n}{M_n} \cdot \mathrm{Radius}(B))=\exp(-\frac{n}{M_n} \cdot \text{Deviation from truth})$.
\end{remark}

\begin{remark} \label{rem:Bayesian-2}
When the function $\g$ is bounded on $\Theta$, \eqref{eq:B-ubound-1} and \eqref{eq:B-lbound-third}, together with the assumptions (B6) that $|\hl(\t)|=O_P(1/\M)$, imply that, to the leading order, roughly speaking we have with high probability
 \begin{equation} \label{eq:rem-B-2}
                     -C_1 |\g(\t)| \log \log n/m_n  \le \frac{1}{n}L_n(\t) \le  - C_2 |\g(\t)|/M_n,  
 \end{equation}
 for some positive constants $C_1$ and $C_2$. It may be noted that, for many natural models, like a location family, $m_n$ and $M_n$ are of the same order as $n \to \infty$. In such cases, \eqref{eq:rem-B-2} provides comparable upper and lower bounds for the log-likelihood $L_n(\t)$, upto $\log \log$ factors.
\end{remark}

\section*{Acknowledgements}
S.G. would like to thank Bikramjit Das for pointing to pertinent references on extreme values, and Philippe Rigollet for helpful discussions. S.C. would like to thank Mark Handcock for helpful comments on ERGM models.
S.G. was supported in part by the MOE grant R-146-000-250-133. S.C. was supported in part by the MOE grant R-155-000-194-114.

\bibliography{ElDeg_article}

%

%


\end{document}